\documentclass[a4   paper,12pt]{article}
\usepackage[cp1251]{inputenc}
\usepackage[shorthands=off,english]{babel}
\usepackage{amsfonts}
\usepackage{tikz, tikz-cd}
\usepackage{amssymb,amsmath,amsthm,eucal,amscd}
\usepackage[left=1cm, right=1cm, top=1cm, bottom=2cm]{geometry}
\usepackage{mathtools}
\usepackage{enumitem}
\usepackage[figurename=Fig.]{caption}
\usetikzlibrary{patterns}
\usepackage{pgfplots}
\usepackage{relsize}
\usetikzlibrary{babel}

\usepackage{amsrefs}

\usepackage{mathtools}
\usepackage{url}
\usepackage{xcolor}
\usepackage{graphicx}

\usepackage{yhmath}
\mathtoolsset{showonlyrefs}
\newcommand{\Z}{\mathbb Z}
\newcommand{\Q}{\mathbb Q}
\newcommand{\R}{\mathbb R}
\newcommand{\C}{\mathbb C}
\newcommand{\I}{\mathbb I}

\newcommand{\N}{\mathbb N}
\newcommand{\mc}{\mathcal}

\newcommand{\lb}{\lbrace}
\newcommand{\rb}{\rbrace}
\newcommand{\la}{\langle}
\newcommand{\ra}{\rangle}
\renewcommand{\phi}{\varphi}

\DeclareMathOperator{\rk}{rk}
\DeclareMathOperator{\conv}{conv}

\DeclareMathOperator{\inter}{int}
\DeclareMathOperator{\cut}{cut}
\DeclareMathOperator{\Ann}{Ann}
\DeclareMathOperator{\cone}{cone}
\DeclareMathOperator{\conconv}{con}
\DeclareMathOperator{\vertex}{vert}
\DeclareMathOperator{\lin}{lin}
\DeclareMathOperator{\ind}{ind}

\renewcommand{\leq}{\leqslant}
\renewcommand{\geq}{\geqslant}

\theoremstyle{plain}
\newtheorem{thm}{Theorem}[section]
\newtheorem{lm}[thm]{Lemma}
\newtheorem{cor}[thm]{Corollary}
\newtheorem{pr}[thm]{Proposition}

\newtheorem{conj}[thm]{Conjecture}
\newtheorem*{prob}{Problem}
\theoremstyle{remark}
\newtheorem{rem}[thm]{Remark}
\newtheorem{ex}[thm]{Example}
\theoremstyle{definition}
\newtheorem{definition}[thm]{Definition}

\tikzset{
  every picture/.append style={
    execute at begin picture={\shorthandoff{"}},
    execute at end picture={\shorthandon{"}}
  }
}

\begin{document}
\date{}

\author{
	Grigory Solomadin\thanks{The work was done at the Steklov Institute of Mathematics RAS and
supported by the Russian Science Foundation, grant 14-11-00414. E-mail: \texttt{grigory.solomadin@gmail.com}}
}
\title{Quasitoric stably normally split manifolds}

\maketitle
\abstract{
A smooth stably complex manifold is called a totally tangentially/normally split manifold (TTS/TNS-manifold, for short, resp.) if the respective complex tangential/normal vector bundle is stably isomorphic to a Whitney sum of complex linear bundles, resp. In this paper we construct manifolds $M$ s.t. any complex vector bundle over $M$ is stably equivalent to a Whitney sum of complex linear bundles. A quasitoric manifold shares this property iff it is a TNS-manifold. We establish a new criterion of the TNS-property for a quasitoric manifold $M$ via non-semidefiniteness of certain higher-degree forms in the respective cohomology ring of $M$. In the family of quasitoric manifolds, this generalises the theorem of J. Lannes about the signature of a simply connected stably complex TNS $4$-manifold. We apply our criterion to show the flag property of the moment polytope for a nonsingular toric projective TNS-manifold of complex dimension $3$.
}

\section{Introduction}

TTS- and TNS-manifolds appeared in the works of Arthan and Bullet \cite{ar-bu-82}, Ochanine and Schwartz \cite{oc-sch-85}, Ray \cite{ra-88} and \cite{so-17} related to a representation of a given complex cobordism class with a manifold from a prescribed family. A naturally arising problem here is to study TTS/TNS-manifolds in well-known families of manifolds, for example, quasitoric manifolds (see \cite{bu-pa-15}, \cite{da-ja-91}). Remind that any quasitoric manifold can be endowed with the natural stably complex structure. The stably complex manifold obtained in this way is a TTS-manifold. (See \cite[Section 7.3]{bu-pa-15} and formula \eqref{eq:tangent} below.) In this way, the problem boils out to the study of quasitoric TNS-manifolds. (In other words, we want to describe TNS-manifolds in a special family of TTS-manifolds.)

Let us remind some facts about quasitoric TNS-manifolds (see \cite{so-17}). The complex projective space $\C P^n$ is a TNS-manifold iff $n<2$. Any invariant submanifold $Z\subset M$ of a quasitoric TNS-manifold $M^{2n}$ is again a TNS-manifold. These two facts together imply that the moment polytope of a quasitoric TNS-manifold has no triangular faces. (A quasitoric manifold over a polytope without triangular faces may be non-TNS, generally speaking, see Example \ref{ex:3-belt}.) If $M$ is a smooth projective toric TNS-manifold of complex dimension $n$ and $Z^{2(n-2)}\subset M^{2n}$ is \textit{any} smooth closed subvariety of complex codimension $2$, then the respective blow-up $Bl_Z M$ of the variety $M$ along $Z$ is a TNS-variety. Any Bott tower is a TNS-manifold. Successive blow-ups of any invariant submanifolds of complex codimension $2$ starting from any Bott tower give different toric TNS-manifolds. Any polytope from famous families of simple polytopes such as flag nestohedra, graph-cubohedra and graph-associahedra admits a realisation as a $2$-truncated cube with the canonical Delzant structure (see \cite{bu-vo-11}, \cite{bu-vo-12}). The respective toric varieties are obtained from any Bott tower (of the necessary dimension) by consequent blow-ups of invariant subvarieties of complex dimension $2$. Consequently, any combinatorial flag nestohedron, graph-cubeahedron or graph-associahedron admits a toric TNS-variety over it. Another operations in the family of stably complex closed TNS-manifolds are cartesian product and connected sum of manifolds. In toric topology there are well-defined equivariant and box sum of any two quasitoric manifolds of the same dimension. We remark that equivariant and box sum of quasitoric TNS-manifolds are also TNS-manifolds.

The complete description of quasitoric TNS-manifolds of dimension $4$ (toric surfaces, in particular) is given by the following Theorem in terms of the signature. (Remind that any quasitoric manifold is simply connected.)

\begin{thm}[\cite{oc-sch-85}, J. Lannes]\label{thm:lannes}
Let $M^{4}$ be a stably complex simply connected closed manifold.

a) If the intersection form of two-dimensional cycles of $M^{4}$ is non-definite then the complex normal bundle of $M^4$ is stably equivalent to the sum $\xi_{1}\oplus\xi_{2}$ for some complex linear bundles $\xi_{1},\xi_{2}\to M^{4}$.

b) If the intersection form of $M^{4}$ is definite, then $M^{4}$ is not a TNS-manifold.
\end{thm}

One can use Theorem \ref{thm:lannes} to show that a smooth projective toric surface with the moment polygone $P^2\subset\R^2$ is a TNS-manifold iff $P^2\neq \Delta^2$, where $\Delta^2$ is the moment polygone of $\C P^2$, i.e. the Delzant triangle. (See Proposition \ref{pr:class}.) However, the TNS-property for quasitoric manifolds is shown to depend not only on the combinatorial type of the moment polytope but also on the respective characteristic matrix in all dimensions greater than $2$. (See Subsection \ref{ssec:4folds}.) Due to that reason, in this paper we focus on the study of toric TNS-manifolds.

For any element $a\in H^{2(n-k)}(M^{2n};K),\ a\neq 0,\ 0<k\leq n$, of the cohomology ring of a quasitoric manifold $M^{2n}$ we define a homogeneous (non-trivial) $k$-form $Q_a: H^2(M^{2n};K)\to K$, by the formula $Q_a(x)=\la a\cdot x^{k},[M^{2n}]\ra$ where $K$ is $\Z,\Q$ or $\R$. This definition is $K$-linear w.r.t. $a$.
\begin{definition}
The form $Q_a$ is called \textit{admissible}, if it is not semi-definite. In other words, the admissible form $Q_a$ has to take values of different signs
\end{definition}

The main result of this paper is as follows.

\begin{thm}\label{thm:crit}
Let $M^{2n}$ be a quasitoric manifold of dimension $2n$. Then $M^{2n}$ is a TNS-manifold iff the $2k$-form $Q_a$ is admissible for any integer $0<k\leq n/2$ and any $a\in H^{2(n-2k)}(M^{2n};\Z),\ a\neq 0$.
\end{thm}

The conditions for odd $k$ are always satisfied so not included in the above Theorem. In the case of even $n$, among these homogeneous forms one has the $n$-form corresponding to the volume polynomial of the multifan of $M^{2n}$ \cite{ay-15}. There is a caveat: for any elements $a,b\in H^{2(n-2k)}(M^{2n};\Z)$ with admissible $2k$-forms $Q_a,Q_b$ the sum $Q_{a+b}$ is not admissible, generally speaking. That is why Theorem \ref{thm:crit} requires admissibility of an infinite set of forms. In the case of $n=3$ we reduce the condition of Theorem \ref{thm:crit} to the study of a finite set of quadratic forms (see Theorem \ref{thm:comp}). Due to that reason one can check the TNS-property for a $6$-dimensional quasitoric manifold using PC. We show that the TNS-property for a quasitoric manifold $M^{2n}$ is equivalent to the total splitting of any complex vector bundle over $M^{2n}$ after stabilisation (Theorem \ref{thm:equiv}). We also study some operations on quasitoric TNS-manifolds, namely: the equivariant blow-up of an invariant submanifold of real codimension $4$ of a given quasitoric TNS-manifold (Proposition \ref{pr:blowup}) and equivariant connected sum of two given quasitoric TNS-manifolds at fixed points. We prove that the equivariant connected sum $M^{2n}_1\sharp M^{2n}_2$ of quasitoric manifolds of dimension $2n$, where $n$ is odd, is a TNS-manifold iff $M_1, M_2$ are TNS-manifolds (Proposition \ref{pr:eqsum}). This claim has an interesting generalisation to the case of even $n$ (ibid.). As an example, we show that the equivariant connected sum of \textit{any} quasitoric $4$-manifold $M^4$ with the fixed quasitoric $4$-manifold $B^4$ at fixed points of opposite signs is a TNS-manifold (Proposition \ref{pr:eqTNS}). (The moment polygone of $B^4$ is rectangular.) We also remark that this claim is easily generalised to the case of arbitrary even $n$.

It is natural to suppose that the TNS-property of a toric manifold depends only on the face lattice of the respective moment polytope, or is even equivalent to the flag property of the latter. These conjectures are discussed in Section \ref{sec:concl}. The cohomology algebra $H^*(M^{2n};\R)$ of a quasitoric manifold $M^{2n}$ is isomorphic to the quotient algebra $DOp_{\R}(\R^m)/\Ann V_{\mc{F}}$ of the algebra of differential operators with constant coefficients $DOp_{\R}(\R^m)$ by the annihilator ideal $\Ann V_{\mc{F}}$ of the volume polynomial $V_{\mc{F}}$ of the multifan $\mc{F}$ of $M^{2n}$ (see Theorem \ref{thm:volcoh}). This fact allows us to reformulate the TNS-criterion for a quasitoric manifold in terms of the respective volume polynomial, see Theorem \ref{thm:volpol}. We also pose a $17$ Hilbert's problem-type question for some finite dimensional $\R$-algeras (see Theorem \ref{thm:algform}) and a give a conjecture about real psd-forms. Finally, we remark that the TNS-criterion for a quasitoric manifold in terms of the respective volume polynomial is algorithmically verifiable, see Corollary \ref{cor:alg}.

\section{TNS-criterion}
The following Section contains some basic tools for the study of a TNS-property for quasitoric manifolds. The TNS-condition of a quasitoric manifold $M^{2n}$ is expressed in terms of the support functions of the cone $S(M^{2n})\subseteq\widetilde{K}^{0}(M^{2n})$ generated by the elements of the form $[\xi]-1$, where $\xi\to M^{2n}$ is a complex linear bundle, in Subsection \ref{ssec:basic}. In the Subsection \ref{ssec:supp} we study this cone for products of complex projective spaces $\C P^n$, then for arbitrary quasitoric manifolds. The necessary definitions are given in Subsection \ref{ssec:conepre}. The name of Subsection \ref{ssec:main} is self-explanatory.

\subsection{TNS-property in terms of $K$-theory and cohomology ring}\label{ssec:basic}


A quasitoric $2n$-dimensional smooth manifold $M^{2n}$ has the canonical stably complex structure
\begin{equation}\label{eq:tangent}
TM^{2n}\oplus\underline{\C}^{m-n}\simeq \bigoplus_{i=1}^{m} \theta_{i}
\end{equation}
for the complex line bundles $\theta_{i}\to M^{2n}, i=1,\dots,m$. (See \cite{bu-pa-15}.) Let $x_{i}:=c_{1}(\theta_{i})$.



\begin{thm}\cite{at-hi-60}\label{thm:lat}
Let $X$ be a finite CW-complex. Then the Chern character map\\
$ch\otimes \Q: K^{*}(X)\otimes \Q\to H^{*}(X;\Q)$ is a $\Z/2\Z$-graded ring isomorphism.
\end{thm}

Let $\Lambda=(\lambda_i^j)$ be the characteristic $(n\times m)$-matrix of the quasitoric manifold $M^{2n}$. Let $P^{n}\subset\R^n$ be the moment polytope of $M^{2n}$ with facets $F_1,\dots,F_m$.

\begin{thm}\cite{da-ja-91}\label{thm:coh}
\[
H^*(M^{2n};\Z)\simeq\Z[x_1,\dots,x_m]/(I_{H}+J_{H}),\ \deg x_i=2,\ i=1,\dots,m,
\]
where $I_{H}=(x_{i_1}\cdots x_{i_k}|\ F_{i_1}\cap\dots\cap F_{i_k}=\varnothing)$, $J_{H}=(\lambda_i^1 x_{1}+\dots +\lambda_i^m x_m|\ i=1,\dots,n)$ are the ideals of the polynomial ring. The ring $H^*(M^{2n};\Z)$ is torsion-free and all respective odd-graded groups vanish.
\end{thm}

Using the Atiyah and Hirzebruch spectral sequence and Theorem \ref{thm:coh} one justifies the following
\begin{pr}\label{pr:ch}
Let $M^{2n}$ be a quasitoric manifold. Then $K(M^{2n}):=K^{0}(M^{2n})$ is a free abelian group (w.r.t. the sum) of rank equal to the Euler characteristic $\chi(M^{2n})$ of $M^{2n}$. Next, $K^1(M^{2n})=0$.
\end{pr}

Remind that
\begin{equation}\label{eq:kcp}
K(\C P^n)\simeq \Z[x]/(x-1)^{n+1},
\end{equation}
where $x$ is a class of the complex conjugate $\bar{\eta}\to \C P^n$ to the tautological line bundle, and $1$ is the class of the trivial linear vector bundle. (See \cite{at-67}.)

\begin{pr}\label{pr:nilp}
Let $\xi\to M^{2n}$ be a complex linear vector bundle over $M^{2n}$. Then in $K(M^{2n})$ one has the relation $([\xi]-1)^{n+1}=0$.
\end{pr}
\begin{proof}
Observe that $ch(([\xi]-1)^{n+1})=0$, then use Theorem \ref{thm:lat}.
\end{proof}

\begin{thm}\cite[Proposition 3.2]{sa-um-07}\label{thm:kdesc}
One has
\[
K(M^{2n})\simeq\Z[\theta_1,\dots,\theta_m]/(I_{K}+J_{K}),
\]
where $\underline{\theta}^{\underline{v}}:=\theta_{1}^{v_{1}}\cdots \theta_{m}^{v_{m}}$ for $\underline{v}=(v_1,\dots,v_m)$. $I_{K}=((\theta_{i_1}-1)\cdots(\theta_{i_k}-1)|\ F_{i_1}\cap\dots\cap F_{i_k}=\varnothing)$, $J_{K}=(\theta_{1}^{\lambda_i^1}\cdots\theta_{m}^{\lambda_i^m}-1|\ i=1,\dots,n)$ are the ideals of the polynomial ring. In particular, the classes $[\theta_{i}]$ multiplicatively generate the ring $K(M^{2n})$, $i=1,\dots,m$.
\end{thm}

\begin{rem}
Theorems \ref{thm:coh} and \ref{thm:kdesc} may also be deduced from the computation of the complex cobordism ring of a quasitoric manifold and further specialisation of the universal formal group law (for complex oriented generalised cohomology theories) to the cohomological or $K$-theoretical, respectively.
\end{rem}

\begin{cor}\label{cor:genpow}
For any $k=1,\dots,n$ one has
\[
H^{2k}(M^{2n};\Q)\simeq\Q\la x^k|\ x\in H^{2}(M^{2n};\Z)\ra
\]
\end{cor}
\begin{proof}
Theorem \ref{thm:kdesc} tells that $K(M^{2n})\simeq\Z\la \underline{\theta}^{\underline{v}}|\ \underline{v}\in\Z^{m}\ra$. Theorem \ref{thm:lat} implies that $ch_k(K(M^{2n})\otimes\Q)=H^{2k}(M^{2n};\Q)$. It remains to observe that $ch_{k}(\underline{\theta}^{\underline{v}})=(c_1(\underline{\theta}^{\underline{v}}))^{k}/k!$.
\end{proof}

%

Consider the semigroup $C(M^{2n})\subseteq \widetilde{K} (M^{2n}):=\widetilde{K}^0(M^{2n})$ generated by the elements of the form $[\xi]-1$, where $\xi\to M^{2n}$ is any linear vector bundle over $M^{2n}$ (with respect to the Whitney sum operation). 

\begin{cor}\label{cor:ccone}
One has
\[
C(M^{2n})=\Z_{\geq 0}\la \underline{\theta}^{\underline{v}}-1|\ \underline{v}\in\Z^{m}\ra,
\]
where $\Z_{\geq 0}\la*\ra$ denotes the $\Z_{\geq}$-semigroup hull of a given abelian group.
\end{cor}


\begin{pr}\cite[Lemma 2.3]{so-17}\label{pr:vx}
Let $X,Y$ be finite CW-complexes. Suppose that $\xi,\eta\to X$ are linear vector bundles with totally split stably inverses. Let $f:Y\to X$ be a continuous map. Then the stably inverse complex vector bundles to  $\overline{\xi},f^*\xi,\xi\oplus\eta,\xi\eta\to X$ are totally split.
\end{pr}

%


\begin{thm}\label{thm:equiv}
Let $M^{2n}$ be a quasitoric manifold of dimension $2n$. Then the following conditions are equivalent:

$(i)$ $M^{2n}$ is a TNS-manifold;

$(ii)$ For any $i=1,\dots,m$, the stably inverse vector bundle to $\theta_{i}$ totally splits;

$(iii)$ For any element $x\in K(M^{2n})$ there exists $N\in\Z$ s.t.
\begin{equation}\label{eq:dec1}
x=N+\sum_{\underline{v}\in\Z^{m}}c_{\underline{v}}[\underline{\theta}^{\underline{v}}],
\end{equation}
with all integers $c_{\underline{v}}\geq 0$ being non-negative (the sum above is over the semigroup: $\underline{v}=(v_{1},\dots,v_{m})\in\Z^{m}$, only finite number of integers $c_{\underline{v}}\in\Z$ are non-zero);

%

$(iv)$ Any complex vector bundle $\xi\to M^{2n}$ stably totally splits.
%
\end{thm}
\begin{proof}
$(i)\Rightarrow(ii)$. By the condition, there exists a totally split complex vector bundle\\
$\alpha=\bigoplus_{i=1}^{b}\alpha_{i}\to M^{2n}$ and an integer $N\in\N$, s.t.
\[
\theta\oplus\alpha\simeq \underline{\C}^{N}.
\]
The claim now follows from the formula \eqref{eq:tangent}. 

$(ii)\Rightarrow(iii)$. There exist complex vector bundles $\xi, \eta\to M^{2n}$, s.t. $x=[\xi]-[\eta]$ (see \cite{at-67}). Next, for the stably inverse vector bundle $\zeta$ to $\eta$, i.e. $[\zeta]+[\eta]=k,\ k\in \Z$, one has
\[
x=[\xi]+[\zeta]-k.
\]
Hence, w.l.g. we may assume that $x=[\xi]$ is a class of a complex linear bundle. By Theorem \ref{thm:kdesc}, the equality \eqref{eq:dec1} holds with ambient coefficients $c_{\underline{v}}$ and $N=0$ in the ring $K(M^{2n})$. It remains to eliminate the negative coefficients in this identity.
Due to Proposition \ref{pr:vx}, for any $\underline{v}\in\Z^{m}$, the linear vector bundle $\underline{\theta}^{\underline{v}}$ has the totally split stably inverse. Hence, for any $\underline{v}\in\Z^{m}$ s.t. $c_{\underline{v}}<0$ there exists $N_{\underline{v}}\in\Z$ with $N_{\underline{v}}+c_{\underline{v}}[\underline{\theta}^{\underline{v}}]$ represented by a class of some totally split vector bundle. Now the desired statement follows.


$(iii)\Rightarrow(iv)$ and $(iv)\Rightarrow(i)$ follow trivially.%
%
%
\end{proof}


The property $(iv)$ from Theorem \ref{thm:equiv} is homotopy invariant for good topological spaces.

\begin{pr}
Let $X,Y$ be homotopy equivalent finite CW-complexes. Suppose that every complex vector bundle over $X$ is stably totally split. Then every complex vector bundle over $Y$ is also stably totally split.
\end{pr}
\begin{proof}
By the definition, there exist continious maps $f:X\to Y,\ g: Y\to X$ s.t.\\
$g\circ f\simeq_{hot} Id_{X}, f\circ g\simeq_{hot} Id_{Y}$, where $\simeq_{hot}$ denotes homotopy equivalence between maps. Let $\xi\to Y$ be a complex vector bundle over $Y$. By the hypothesis, there exist a totally split complex vector bundle $\alpha=\bigoplus_{i=1}^{k} \alpha_{i} \to X$ s.t.
\[
f^*(\xi)\oplus\underline{\C}^N=\alpha
\]
for some $N\in \N$. Then one has:
\[
g^*(f^*(\xi)\oplus g^*(\underline{\C}^{N}))=(f\circ g)^*(\xi)\oplus\underline{\C}^{N}=g^*(\alpha).
\]
The vector bundle $g^*(\alpha)=\bigoplus_{i=1}^{k} g^*(\alpha_i)$ is totally split. The vector bundles $(fg)^*(\xi)$ and $\xi$ are topologically equivalent, because they have the homotopy equivalent classifying maps. Hence,
\[
\xi\oplus\underline{\C}^N=g^*(\alpha),
\]
and $\xi$ is stably totally split vector bundle. Q.E.D.
\end{proof}


\subsection{Convex cones and operations on them}\label{ssec:conepre}

Recall some standard definitions of convex geometry. From now and on in this paper we consider convex cones only in finite-dimensional real linear spaces.

\begin{definition}
A convex cone $\sigma$ in $\R^n$ (with apex at the origin $0\in\R^n$) is any set of the form $\conconv X:=\lb t_1 v_1+\dots+t_k v_k|\ k\in\N,\ t_1,\dots,t_k\in\R_{\geq 0};\ v_1,\dots,v_k\in X\rb$, where $X\subset\R^n$. The maximal w.r.t. inclusion linear subset of the cone $\sigma$ is called a lineality subspace $\lin \sigma\subset\R^n$ of $\sigma$. The cone $\sigma$ is closed/proper, if it is closed/proper as a subset of $\R^n$, resp., and salient, if $\lin\sigma=0$. The dimension $\dim \sigma$ of the cone $\sigma$ is the dimension of its linear hull $\dim\R\la\sigma\ra$. In case of $\dim\sigma=n$, the cone $\sigma$ is called full-dimensional.
\end{definition}

The following Proposition is clear.
\begin{pr}\label{pr:salcon}
For any convex cone $\sigma\subset\R^n$ there exists a salient cone $\sigma'\subset\R^n$ s.t. one has $\sigma=\sigma'+\lin \sigma$.
\end{pr}

Fix any basis $e_1,\dots,e_n\in\R^n$.

\begin{lm}\label{lm:supp}
Let $\sigma\subset\R^n$ be a full-dimensional salient cone. Then $\sigma$ has a supporting hyperplane having the normal with only rational coordinates in the basis $e_1,\dots,e_n$.
\end{lm}
\begin{proof}
The normal of any supporting hyperplane to $\sigma$ is an element of the dual cone $\sigma^*\subset (\R^n)^*$. The cone $\sigma$ is salient, hence, $\inter (\sigma^*)\neq\varnothing$. The desired normal is any element of the interior $\inter (\sigma^*)$ having all rational coordinates in the dual basis $e^1,\dots,e^n\in(\R^n)^*$.
\end{proof}

\begin{definition}
A linear subspace $U\subset\R^n$ is called rational w.r.t. the basis $e_1,\dots,e_n\in\R^n$, if $U$ is generated by vectors $u_1,\dots,u_k$ of $U$ having only rational coordinates in the basis $e_1,\dots,e_n$.
\end{definition}

\begin{pr}\label{pr:supplin}
Let $\sigma\subseteq \R^n$ be a full-dimensional proper convex cone with rational lineality subspace $\lin \sigma$ w.r.t. the basis $e_1,\dots,e_n\in\R^n$. Then $\sigma$ has a supporting hyperplane with the normal having only rational coordinates in the given basis of $\R^n$.
\end{pr}
\begin{proof}
Follows immediately from Propositions \ref{pr:salcon} and \ref{lm:supp}.
\end{proof}

The next step is to set up two different definitions of products for convex cones. The first one is the tensor product of cones (see \cite{bo-co-gu-13}), and the second corresponds to the cartesian product of manifolds (see Corollary \ref{cor:cart}).

\begin{definition}\label{def:cone}
Let $\sigma\subseteq U,\tau\subseteq V$ be convex cones in $\R$-linear spaces $U,V$, respectively. The convex cone
\[
\sigma\otimes\tau:=\conconv\lb u\otimes v|\ u\in U,v\in V\rb\subseteq U\otimes V,
\]
is called a \emph{tensor product of cones} $\sigma, \tau$. For any $u\in U,v\in V$ define\\
$u*v:=u+v+u\otimes v$. Also define the product of convex cones by the formula
\[
\sigma*\tau:=\conconv\lb u*v|\ u\in U,v\in V\rb\subseteq U\oplus V\oplus U\otimes V.
\]
\end{definition}
By the definition, $(\sigma*\tau,1)=(\sigma,1)\otimes(\tau,1).$ Observe that the $*$-product is commutative and associative but not linear by either of the factors, generally speaking.

\begin{lm}\label{lm:tensspace}
Let $\sigma=\lin \sigma + \sigma'\subseteq U$, where $\sigma,\sigma'\subseteq U,\ \tau\subseteq V$ are convex cones. Then one has $\sigma\otimes\tau=\lin\sigma \otimes\R\la \tau\ra+\sigma'\otimes\tau$.
\end{lm}
\begin{proof}
An equality $\sigma\otimes\tau= (\lin \sigma)\otimes\tau+\sigma'\otimes\tau$ clearly takes place. It remains to show that\\
$(\lin \sigma)\otimes\tau=\lin \sigma\otimes\R\la\tau\ra$. The inclusion $(\lin \sigma)\otimes\tau\subseteq (\lin \sigma)\otimes\R\la\tau\ra$ clearly holds. The inverse inclusion $(\lin \sigma)\otimes\tau\supseteq (\lin \sigma)\otimes\R\la\tau\ra$ follows from the identities $(\lin \sigma)\otimes\tau= (-\lin \sigma)\otimes\tau=(\lin \sigma)\otimes(-\tau)$, $\R\la\tau\ra=\tau-\tau$.
\end{proof}

\begin{pr}\label{pr:tenscone}
For any convex cones $\sigma\subset U,\ \tau\subset V$ one has
\[
\lin (\sigma\otimes\tau)=(\lin \sigma)\otimes\R\la\tau\ra+\R\la\sigma\ra\otimes\lin\tau.
\]
In particular, the tensor product of two salient convex cones is salient.
\end{pr}
\begin{proof}
First we check the claim about salient cones. Consider supporting functions $H\in U^*,\ L\in V^*$ of the cones $\sigma,\tau$, resp., such that for any non-zero elements $u\in \sigma,\ v\in\tau$ one has $H(u),L(v)>0$. Define the linear function $H\otimes L\in (U\otimes V)^*$ by the formula $(H\otimes L)(u\otimes v):=H(u)\cdot L(v)$, and extend it further by linearity. Then for any non-zero element $w\in \sigma\otimes\tau$ clearly one has $(H\otimes L)(w)>0$, as required. In the general case, the inclusion $\lin (\sigma\otimes\tau)\supseteq(\lin \sigma)\otimes\R\la\tau\ra+\R\la\sigma\ra\otimes\lin\tau$ follows from Lemma \ref{lm:tensspace}. To prove the inverse inclusion, w.l.g. we may assume that $\sigma,\tau$ are salient. It remains to use the proven salient property of the cone $\sigma\otimes\tau$.
\end{proof}

\begin{cor}\label{cor:ratprod}
Let $\sigma\subseteq U,\ \tau\subseteq V$ be full-dimensional convex cones with rational lineality subspaces $\lin \sigma,\ \lin \tau$ w.r.t. the bases $u_1,\dots,u_k$ and $v_1,\dots,v_l$ of linear spaces $U,V$, respectively. Then the subspace $\lin(\sigma\otimes\tau)\subseteq U\otimes V$ is rational w.r.t. the basis $u_1\otimes v_1,\dots,u_k\otimes v_l$ of the linar space $U\otimes V$.
\end{cor}

We need some facts about cyclic polytopes (see \cite{zieg-95}). Let ${\bf x}_n:\R \to \R^{n}, {\bf x}_n(t):=(t,t^2,\dots,t^{n})$. The image of the real line $\R$ under the map ${\bf x}_n$ is called a moment curve. For any $k>n$ the cyclic polytope $C^n (t_1,\dots,t_k)$ is defined as a convex hull of $k$ distinct points ${\bf x}_n(t_1),\dots,{\bf x}_n(t_k),\ t_1,\dots,t_k\in\R$, of the moment curve.

\begin{thm}\cite{zieg-95}
$(i)$ Cyclic polytope $C^n (t_1,\dots,t_k)$ is a simplicial $n$-polytope;

$(ii)$ $C^n (t_1,\dots,t_k)$ has exactly $k$ vertices;

$(iii)$ The combinatorial type of $C^n (t_1,\dots,t_k)$ does not depend on the choice of $t_1,\dots,t_k$.
\end{thm}

For any $k$ let $C^{n}_{k}:=C^n (-1,-1/2,\dots,-1/k,1/k,\dots,1/2,1)$. Also let $C^n_{\infty}$ be the closure of the convex hull of the points $\lb {\bf x}_n(1/k)|\ k\in \Z ,\ k\ne 0\rb\cup\lb {\bf x}_n(0)\rb$ (see Fig. \ref{fig:cyc}).

\begin{cor}\label{cor:vert}
$C^1_{\infty}=[-1,1]$. For $n\geq 2$, the vertices of $C^n_{\infty}$ are $\lb {\bf x}(1/k)|\ k\in \Z\rb\cup\lb {\bf x}(0)\rb$. $C^n_{\infty}=\overline{\bigcup_{k=1}^{\infty} C^{n}_{k}}$ is a compact convex body in $\R^n$.
\end{cor}
\begin{proof}
The set $C^n_{\infty}$ is closed. It remains to notice that $C^{n}_{k}\subset C^{n}_{k+1}$ and $C^n_{\infty}\subset \I^n$, where $\I^n=\lb (x_1,\dots,x_n)\in\R^n| -1\leq x_i\leq 1,\  i=1,\dots,n\rb$ is the unitary $n$-hypercube.
\end{proof}
%
%
%
%
%
%
%

\begin{figure}
\centering
\caption{The polytope $C^3_{50}$.\label{fig:cyc}}
\scalebox{1}{
\includegraphics{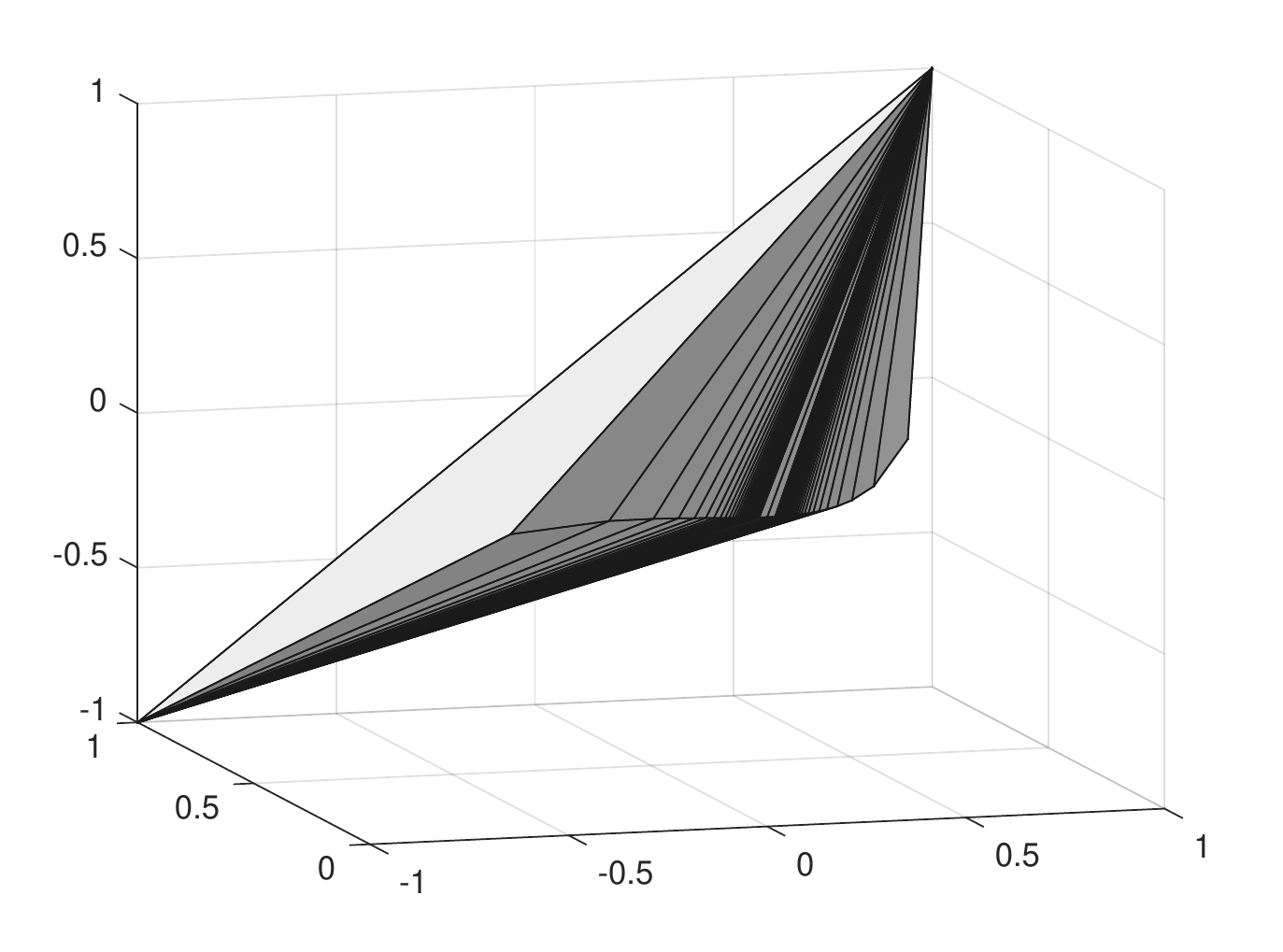}
}
\end{figure}

\subsection{The cone $S(M^{2n})$}\label{ssec:supp}

Denote by $S(M^{2n})\subseteq\widetilde{K}(M^{2n})\otimes\R$ the closure of the convex conical hull of the semigroup\\
$C(M^{2n})\subset\widetilde{K}(M^{2n})$, i.e. $S(M^{2n})=\overline{\conconv C(M^{2n})}$. In other words, $S(M^{2n})$ is the closure of the conical convex hull of elements of the form $[\xi]-1$, where $\xi\to M^{2n}$ is a complex linear bundle. In this Subsection we give a description of $S(M^{2n})$ for an arbitrary quasitoric manifold $M^{2n}$.

%
%
%
%


The natural projection $\widetilde{K}(\C P^{n})\otimes\R\to \widetilde{K}(\C P^{n-1})\otimes\R$ maps the cone $S(\C P^{n})$ onto the cone $S(\C P^{n-1})$. Fix the basis $(x-1),\dots,(x-1)^n$ of the linear space $\widetilde{K}(\C P^n)\otimes\R\simeq \R^n$ and let $e^1,\dots,e^n$ be the corresponding coordinates in $\widetilde{K}(\C P^n)\otimes\R$. Let $A_n:\widetilde{K}(\C P^n)\otimes\R\to \widetilde{K}(\C P^n)\otimes\R$ be the matrix of the linear coordinate change defined uniquely by the conditions
\[
A_n(\binom{k}{1},\binom{k}{2},\dots, \binom{k}{n})= (k,k^2,\dots,k^n),\ k\in\Z,
\]
in the coordinates indicated above. From now and on in this paper, $\binom{a}{b}:=\frac{a(a-1)\cdots(a-b+1)}{b(b-1)\cdots 1}$ for $a\in \Z,\ b\in\Z_{> 0}$; $\binom{a}{0}:=1$ for $a>0$; $\binom{0}{b}:=0$ for $b\geq 0$. The matrix $A_n$ is well-defined for any $n\in\N$, since the polynomials $\binom{k}{1},\binom{k}{2},\dots, \binom{k}{n}$ in $k$ span an $n$-dimensional linear subspace of the finite-dimensional real polynomial algebra $\R[k]$ ($n$ is fixed). 

\begin{pr}\label{pr:scone}
$S(\C P^{2})=\R_{\geq 0}\la (x-1),(\overline{x}-1)\ra$. If $n$ is odd, then\\
$S(\C P^{n})=\R_{\geq 0}\la (x-1)^n,-(x-1)^n,S(\C P^{n-1})\ra$. If $n$ is even, then the image $A_n(S(\C P^{n}))$ is the cone over the compact convex body $P^{n-1}\subset\lb e^n=1\rb$ of dimension $n-1$. Under the natural linear projection $\lb e^n=1\rb\to\R\la (x-1),\dots,(x-1)^n\ra$, the body $P^{n-1}$ maps bijectively onto $C^{n-1}_{\infty}$. In particular, the linear subspace $\lin S(\C P^n)\subseteq\widetilde{K}(\C P^n)\otimes\R$ is rational w.r.t. the above indicated basis of $\widetilde{K}(\C P^n)\otimes\R$.
\end{pr}
\begin{proof}
Due to the isomorphism \eqref{eq:kcp} and Taylor expansion, one has:
\begin{equation}\label{eq:taylor}
x^k-1=\sum_{i=1}^{n}\binom{k}{i}(x-1)^i,\ k\in\Z.
\end{equation}
By the Formula \ref{eq:taylor}, the cone $A_n(S(\C P^n))$ is generated by the vectors $(k,k^2,\dots,k^n)\in\R^n,\ k\in\Z$.

Let $n$ be odd. To prove the statement it is enough to check that any supporting function of the cone $S(\C P^{n})$ vanishes on $(x-1)^n$. Consider the linear function $H:\widetilde{K}(\C P^n)\otimes\R\to\R$ s.t. $H(S(\C P^{n}))\geq 0$. Due to the identity \eqref{eq:taylor}, one has $0\leq\lim_{k\to +\infty}H(x^k-1)= \lim_{k\to +\infty} k^n H((x-1)^n)$. Hence, $H((x-1)^n)\geq 0$. Next, one has $\lim_{k\to -\infty}H(x^k-1)= \lim_{k\to -\infty} k^n H((x-1)^n)$. We conclude that $H((x-1)^n)=0$, as required.

Let $n$ be even. Dividing the generators $(k,k^2,\dots,k^n), k\in\Z$, by $k^n,k\neq 0$, we see that $A_n(S(\C P^{n}))$ is the cone over the convex body $C^{n-1}_{\infty}$ in the respective affine hyperplane. The explicit formula for $S(\C P^2)$ follows from Corollary \ref{cor:vert}.
%
\end{proof}

The following observation belongs to A. Ayzenberg.
\begin{rem}
For any integer $k$ the expansion coefficients of $x^k-1\in \widetilde{K}(\C P^n)$ w.r.t. the basis $(x-1),...,(x-1)^n$ are the binomial coefficients $(\binom{k}{1},\binom{k}{2},\dots, \binom{k}{n})$. A straight-forward induction on $k$ shows that for any non-negative integer $l$ the inequality
\[
\binom{k}{l}^2\geq \binom{k}{l-1}\binom{k}{l+1}
\]
holds. When $0<l<k$, it is a well-known inequality on the binomial coefficients. On the other hand, this inequality means that the sequence $(\binom{k}{1},\binom{k}{2},\dots, \binom{k}{n})$ is log-concave. This property of the projective space may be worth studying due to thesis \cite{ju-14}.
\end{rem}


\begin{lm}\label{lm:*id}
Let $x_1,\dots,x_k\in {K}(M^6), k\geq 2$. Then one has
\[
x_1*x_2*x_3*\dots*x_k=\sum_{q=1}^k \sum_{1\leq i_1<\dots<i_q\leq k} x_{i_1}\cdots x_{i_q},
\]
\begin{equation}\label{eq:*mult}
(x_1-1) *\cdots * (x_k-1)=x_1\cdots x_k-1.
\end{equation}
\end{lm}

The following Lemma is clear.
\begin{lm}\label{lm:change}
The transition matrix from the basis $\lb (y_1-1)^{v_1}\cdots(y_m-1)^{v_m}|\ \sum_i v_i\leq n,\ v_i\in\Z_{\geq 0}\rb$ to the basis $\lb y_1^{v_1}\cdots y_m^{v_m}-1|\ \sum_i v_i\leq n,\ v_i\in\Z_{\geq 0}\rb$ of the linear space $\widetilde{K}((\C P^n)^m)\otimes\R$ has only rational matrix elements.
\end{lm}


\begin{pr}\label{pr:scpmultdesc}
For any $n,m\in\N$ one has
\begin{equation}\label{eq:scpmult}
S((\C P^{n})^m)=S(\C P^{n})*\dots *S(\C P^{n}).
\end{equation}
The subspace $\lin S((\C P^{n})^m)\subseteq \widetilde{K}((\C P^{n})^m)$ is rational w.r.t. the basis\\
$\lb y_1^{v_1}\cdots y_m^{v_m}-1|\ \sum_i v_i\leq n,\ v_i\in\Z_{\geq 0}\rb$ of the linear space $\widetilde{K}((\C P^n)^m)\otimes\R$.
\end{pr}
\begin{proof}
The Formula \eqref{eq:scpmult} is a straight consequence of the Formula \eqref{eq:*mult} and K\"{u}nneth formula for $K$-theory (see \cite{at-67}). The second claim follows from the Corollary \ref{cor:ratprod}, Proposition \ref{pr:scone} and Lemma \ref{lm:change}.
\end{proof}

Let $M^{2n}$ be a quasitoric manifold with the convex polytope with $m$ facets. Consider the linear map $R: K((\C P^n)^m)\to K(M^{2n})$ mapping the class $y_i$ of the dual to the (pull-back of the) tautological line bundle over the $i$-th multiple in $(\C P^n)^m$ to $\theta_i$. This map is well-defined due to Proposition \ref{pr:nilp} and Theorem \ref{thm:kdesc}. By the definition, $R(y_1^{v_1}\cdots y_m^{v_m})=\underline{\theta}^{\underline{v}}$ holds for any $\underline{v}\in\Z^m$. Hence, $S(M^{2n})=R(S((\C P^n)^m))$ (see Corollary \ref{cor:ccone}).



\begin{lm}\label{lm:basis}
There exist complex linear vector bundles $\xi_i\to M^{2n},\ i=1,\dots,\chi(M^{2n})-1$, such that $e_{1}-1,\dots,e_{\chi(M^{2n})-1}-1$ constitute a basis of the free abelian group $\widetilde{K}(M^{2n})$, where $e_i=[\xi_i],\ i=1,\dots,\chi(M^{2n})-1$. Given any complex linear vector bundle $\xi\to M^{2n}$, the element $[\xi]-1\in \widetilde{K}(M^{2n})$ has rational coordinates in the basis above.
\end{lm}
\begin{proof}
Follows from Theorem \ref{thm:kdesc} and multivariate Taylor formula for $f(v_1,\dots,v_m)=\underline{\theta}^{\underline{v}}-1$.
\end{proof}

\begin{cor}\label{cor:rat}
Suppose that $S(M^{2n})\neq \widetilde{K}(M^{2n})\otimes\R$. Then the cone $S(M^{2n})$ has a supporting hyperplane with the normal vector having only rational coordinates in the basis from Lemma \ref{lm:basis}.
\end{cor}
\begin{proof}
Theorem \ref{thm:kdesc} and Lemma \ref{lm:basis} imply that the matrix of the linear map $R$ has only rational entries in the above bases of the linear spaces $\widetilde{K}((\C P^n)^m)\otimes\R,\ \widetilde{K}(M^{2n})\otimes\R$. It follows from Proposition \ref{pr:scpmultdesc} and Lemma \ref{lm:basis} that the subspace $\lin S(M^{2n})\subseteq\widetilde{K}(M^{2n})\otimes\R$ is rational w.r.t. the above basis. Now the claim follows from Proposition \ref{pr:supplin}.
\end{proof}

\begin{cor}\label{cor:cart}
For any quasitoric manifolds $M_1^{2n_1}, M_2^{2n_2}$ an identity
\[
S(M_1^{2n_1}\times M_2^{2n_2})=S(M_1^{2n_1})*S(M_2^{2n_2})
\]
holds under the K\"{u}nneth isomorphism $K(M_1^{2n_1}\times M_2^{2n_2})\simeq K(M_1^{2n_1})\otimes K(M_2^{2n_2})$.
\end{cor}

\subsection{Proof of the main theorem}\label{ssec:main}

In order to prove the TNS-criterion we need an auxiliary



\begin{lm}\label{lm:conv}
Let $S\subseteq L$ be a sub-semigroup of the free abelian group $L\simeq\Z^d$ s.t. $S$ contains a $\Z$-basis $x_{1},\dots,x_{d}\in S$ of $L$. Then the following conditions are equivalent:

$(i)$ $S=L$;

$(ii)$ There exists an element $v=\sum_{i}v^{i}x_{i}\in S$ s.t. $v^{i}< 0,\ i=1,\dots,d$;

$(iii)$ $0\in\inter \conv S$, where $\inter$ and $\conv$ denote the interior and the convex hull of a set in $\R^d$, resp.;

$(iv)$ There is no such a linear function $H:L\otimes\R\to\R,H\not\equiv 0$ that $S\subseteq\lb H\geq 0\rb$.
\end{lm}
\begin{proof}
The implications $(i)\Rightarrow (ii)\Rightarrow (iii)\Rightarrow (iv)$ are clear. $(iv)\Rightarrow (iii)$ follows from the Supporting Hyperplane Theorem in $\R^n$.

$(iii)\Rightarrow (ii)$. The condition $(iii)$ implies that there exists such an element $u=\sum_{i}u^{i}x_{i}\in\inter\conv S$ that $u^{i}<0,i=1,\dots,d$. Consider elements $u_{1},\dots,u_{d}\in S$ and real positive numbers $a_{1},\dots,a_{d}\in\R_{\geq 0}$ s.t. $u=\sum_{i}a^{i}u_{i}$. Consider a small variation $b^{i}$ of $a^{i},i=1,\dots,d$, s.t. $u':=\sum_{i}b^{i}u_{i}\in\inter\conv S$, all coordinates (w.r.t. the basis $x_{1},\dots,x_{d}$) of $u'$ are negative and all $b^{i}\in\Q_{> 0}$ are rational positive numbers. Let $N\in \N$ be s.t. $Nb^{i}\in\Z_{\geq 0},i=1,\dots,d$. Then all the coordinates of $v:=Nu'\in S$ are negative (w.r.t. the basis $x_{1},\dots,x_{d}$), as required.

$(ii)\Rightarrow (i)$. Let $x\in L$. Consider a decomposition $x=\sum_{i}a^{i}x_{i}$, $a_{i}\in\Z$, $i=1,\dots,d$. Let $L_{\geq 0}\subseteq S$ be a sub-semigroup generated by $x_{1},\dots,x_{d}$. Let $N\in\N$ be a natural number s.t. $Nv< x$ (coordinate-wisely). Then $x\in Nv+L_{\geq 0}$. Hence $x\in S$, as required.
\end{proof}

\begin{cor}\label{cor:app}
A quasitoric manifold $M^{2n}$ of dimension $2n$ is a TNS-manifold iff $S(M^{2n})\neq \widetilde{K}(M^{2n})\otimes\R$.
\end{cor}
\begin{proof}
The TNS-property of $M^{2n}$ is equivalent to the condition $(iii)$ of Theorem \ref{thm:equiv}. The latter is, in turn, equivalent to the condition
\begin{equation}\label{eq:kcone}
\widetilde{K}(M^{2n})=C(M^{2n}).
\end{equation}
By Lemma \ref{lm:basis}, the semigroup $C(M^{2n})=\Z_{\geq 0}\la \underline{\theta}^{\underline{v}}-1|\ \underline{v}\in\Z^{m}\ra$ contains a $\Z$-basis of the free abelian group $\widetilde{K}(M^{2n})$. Hence, by Lemma \ref{lm:conv}, the equality \eqref{eq:kcone} is equivalent to the desired condition.
\end{proof}

Define the homogeneous $\Q$-form $Q_{a}: H^2(M^{2n};\Q)\to\Q$ of degree $k$ by the formula
\[
Q_{a}(x):=\la x^k a, [M^{2n}]\ra,\ x\in H^{2}(M^{2n};\Q),
\]
where $\la*,*\ra$ is the canonical pairing.

\begin{rem}
For a quasitoric manifold $M^{2n}$, the volume polynomial of the multifan $\mc{F}$ corresponding to $M^{2n}$ is given by the formula
\[
V_{\mc{F}}(c_1,\dots, c_m):=\frac{1}{n!}\la (c_1 x_1+\dots +c_m x_m)^n, [M^{2n}]\ra,\ c_i\in \R,\ i=1,\dots, m.
\]
An identity $Q_1(c_1 x_1+\dots +c_m x_m)=\frac{1}{n!}V_{\mc{F}}(c_1,\dots,c_m)$ clearly holds. Notice that the form $Q_a$ may be degenerate (see Example \ref{ex:cp2cp2}).

\end{rem}

\begin{pr}\label{pr:corr}
Let $K=\Z,\Q$ or $\R$. Then there is a bijective correspondence between $K$-linear functions $L: H^{2k}(M^{2n};K)\to K$ and homogeneous $k$-forms $Q_a: H^{2}(M^{2n};K)\to K$, $a\in H^{2(n-k)}(M^{2n};K)$. 
\end{pr}
\begin{proof}
We give the proof for the case of $K=\Q$. The Poincar\'e duality implies that any $\Q$-linear function $L: H^{2k}(M^{2n};\Q)\to \Q$ has the form $L(*)=\la *\cdot a,[M^{2n}]\ra$ for some $a\in H^{2(n-k)}(M^{2n};\Q)$. Hence, Corollary \ref{cor:genpow} gives the required bijective correspondence.
\end{proof}

\begin{proof}[Proof of Theorem \ref{thm:crit}]
$\Rightarrow$. Assume the contrary. Then there exists $a\in H^{2(n-2k)}(M^{2n};\Z)$ s.t. the form $Q_a$ is non-trivial and semidefinite. Let $x\in H^{2}(M^{2n};\Z)$ be s.t. $Q_a(x)> 0$. Consider a complex linear vector bundle $\xi\to M^{2n}$ with $c_1(\xi)=x$. The stably inverse vector bundle $\alpha\to M^{2n}$ to $\xi$ is totally split by Theorem \ref{thm:equiv}: $\alpha=\bigoplus_{i=1}^{l}\alpha_i$ for some $\alpha_i\to M^{2n}$. Let $a_i:=c_1(\alpha_i)$. Apply the Chern character map to the identity $l+1-\xi=\alpha$ in $K(M^{2n})$. The respective $4k$-component of the obtained identity is
\[
-x^{2k}=\sum_{i=1}^{l}a_i^{2k}.
\]
Now multiply the left and right parts of the last identity by $a$, then couple the obtained identity with $[M^{2n}]$ to get
\[
-Q_a(x)=\sum_{i=1}Q_a(a_i).
\]
But the latter contradicts the semidefiniteness of $Q_a$.

%
%
$\Leftarrow$. Assume the contrary. Then Corollary \ref{cor:rat} and Lemma \ref{lm:basis} imply that there exists a linear function $H: K(M^{2n})\otimes\R\to \R$ s.t. $H\not\equiv 0,\ H(1)=0$ and $H(\underline{\theta}^{\underline{v}})\in\Q_{\geq 0}$ for any $\underline{v}\in\Z^m$. Define the linear function $L:\ H^*(M^{2n};\R)\to\R$ by the formula
\[
L(x):=H(ch^{-1}(x)+\overline{ch^{-1}(x)}),\ x\in H^*(M^{2n};\R),
\]
where the bar denotes complex conjugation in $K$-theory. Notice, that $L(1)=0$ and $L|_{H^{2(2k+1)}(M^{2n};\R)}\equiv 0$ for any $k=0,\dots,[n/2]$. Suppose that $L\equiv 0$. Substituting $x=ch(\underline{\theta}^{\underline{v}})$ one obtains
\[
H(\underline{\theta}^{\underline{v}})=-H(\underline{\theta}^{-\underline{v}}),
\]
for any $\underline{v}\in\Z^m$. Then $H\equiv 0$ --- a contradiction. Hence, $L\not\equiv 0$.

Let $k$ be the greatest integer s.t. $L|_{H^{2k}(M^{2n};\R)}\not\equiv 0$. It follows from the definition of Chern character and $L\not\equiv 0$ that $k>0$ is even. Due to Proposition \ref{pr:corr}, the linear function $L|_{H^{2k}(M^{2n};\Q)}$ gives a non-zero $k$-form $Q:\ H^{2}(M^{2n};\Q)\to\Q$ of even degree. Using its homogenity w.l.g. we may assume that $Q$ is integer. By the condition, this form is non-semidefinite. Hence, there exists an element $x\in H^2(M^{2n};\Z)$ s.t. $Q(x)=L(x^k)<0$. Let $\xi\to M^{2n}$ be a complex linear vector bundle s.t. $c_1(\xi)=x$. Then one has
\[
0\leq H(\xi^a+\xi^{-a})=L(ch(\xi^a))=\sum_{i=1}^{k/2} \frac{a^{2i}}{(2i)!} L(x^{2i}),\ a\in\Z.
\]
Hence,
\[
0\leq \lim_{a\to+\infty}L(ch(\xi^a))=L(x^k)\cdot(+\infty)=-\infty
\]
--- a contradiction. Q.E.D.

\end{proof}

\begin{rem}
$\Q$- and $\R$-analogues of Theorem \ref{thm:crit} clearly take place.
\end{rem}

\begin{ex}\label{ex:cp2cp2}
Let $M^8=\C P^2\times \C P^2$. The respective cohomology ring is $H^*(M^8;\Q)\simeq\Q[x,y]/(x^3,y^3)$, where $x,y$ are the first Chern classes of the pull-backs of the dual to the tautological bundles over the respective factors in $\C P^2\times \C P^2$. We apply Theorem \ref{thm:crit} to show that $M^8$ is not a TNS-manifold. The $4$-form $Q_1$ is positive semidefinite: $Q_1(ax+by)=6a^2b^2$, where $a,b\in\Q$. (Notice that the intersection form of $M^8$ is non-definite: $\sigma(M^8)=1<3=\dim H^{2}(M^8;\R)$.) Any element of $H^4(M^8;\Q)$ has the form $ax^2+bxy+cy^2$, where $a,b,c\in\Q$. The quadratic form $Q_{ax^2+bxy+cy^2}$ has matrix
\[
\begin{pmatrix}
c & b\\
b & a
\end{pmatrix}
\]
in the basis $x,y\in H^2(M^8;\Q)$. Clearly, $Q_{x^2+y^2}$ is a positive-definite form.
\end{ex}

\begin{cor}\label{cor:powers}
Let $M^{2n}$ be a quasitoric manifold of dimension $2n$. Then $M^{2n}$ is a TNS-manifold iff for any $0<k\leq n/2$ $H^{2k}(M^{2n};\R)=ch_{2k}(S(M^{2n}))=\R_{\geq 0}\la x^{2k}|\ x\in H^{2}(M^{2n};\Z)\ra$.
\end{cor}


\section{TNS-manifolds in low dimensions}
The following Section is devoted to the study of the quasitoric TNS-manifolds in dimensions $4,6$. A reduction of Theorem \ref{thm:crit} to finitely many quadratic forms for $6$-folds is given in Subsection \ref{ssec:6folds}.
\subsection{Quasitoric TNS $4$-folds}\label{ssec:4folds}

%

Here is a complete characterisation of smooth projective toric TNS-surfaces.

\begin{pr}\label{pr:class}
Let $M^4$ be a smooth projective toric surface with the moment polygone $P^2\subset \R^2$. Then $M^4$ is a TNS-manifold iff $P^2$ is distinct from the triangle $\Delta^2$.
\end{pr}
\begin{proof}
Remind that the polygone $P^2$ has $m$ edges. Due to Theorem \ref{thm:lannes} one has to check that the equality $|\sigma(M^4)|=\dim H^2(M^4;\R)$ holds iff $P^2=\Delta^2$. Due to the well-known formula of the signature and Euler characteristic of a toric manifold (e.g. \cite[Section 9.5]{bu-pa-15}), one has $|\sigma(M^4)|=|4-\dim H^2(M^4;\R)|\leq \dim H^2(M^4;\R),\ \dim H^2(M^4;\R)=m-2$. Clearly, the equality holds here iff $m=3$. It remains to notice that the only smooth projective toric surface over the triangle is $\C P^2$.
\end{proof}

Quasitoric non-TNS manifolds are more diverse starting from dimension $4$ and so on. A straight-forward computation implies the following

\begin{pr}\label{pr:nontns}
Let $M^4$ be a quasitoric non-TNS $4$-fold. Suppose that the moment polygone of $M^4$ is a $4$-gon. Then the characteristic matrix of $M^4$ is $GL_{2}(\Z)$-equivalent to
\[
\begin{pmatrix}
Id_2 & A
\end{pmatrix},
\]
where $A$ is the one of the following matrices:
\[
\begin{pmatrix}
1 & 2\\
1 & 1
\end{pmatrix},\
\begin{pmatrix}
1 & -2\\
1 & -1
\end{pmatrix},\
\begin{pmatrix}
1 & -2\\
-1 & 1
\end{pmatrix},\
\begin{pmatrix}
1 & 2\\
-1 & -1
\end{pmatrix}.
\]
These manifolds have two different oriented diffeomorphism classes.
\end{pr}

\subsection{Quasitoric TNS $6$-folds}\label{ssec:6folds}

Consider a quasitoric manifold $M^6$.

\begin{pr}\label{pr:poly}
The cone $S(M^6)$ is polyhedral.
\end{pr}
\begin{proof}
Follows immediately from Propositions \ref{pr:scone}, \ref{pr:scpmultdesc} and surjectivity of the map $R$.
\end{proof}

The analogue of Proposition \ref{pr:poly} in higher dimensions does not hold, generally speaking. Due to Corollary \ref{cor:cart}, it is enough to give an example of a $8$-quasitoric manifold $M^8$ with non-polyhedral cone $S(M^8)$.

\begin{ex}
Let $M^8=\C P^2\times\C P^2$. In the denotations of Example \ref{ex:cp2cp2},\\
$ch_2(S(M^8))=\R_{\geq 0}\la (ax+by)^2|\ a,b\in\R\ra\subset H^4(M^8;\R)$. The latter cone becomes the cone $\sigma=\R_{\geq 0}\la (1,0,0),(t^2,t,1)|\ t\in\R\ra$ after the suitable coordinate change in $H^4(M^8;\R)\simeq\R^3$. The cone $\sigma$ is not polyhedral. Hence, $S(M^8)$ is not polyhedral, as well.
\end{ex}

\begin{lm}\label{lm:chid}
For any classes $\alpha_1,\dots,\alpha_k\in K(M^6)$ of complex linear vector bundles over $M^6$ the identity
\[
ch_2((\alpha_{1}-1)*\dots*(\alpha_{k}-1))=\frac{1}{2}(\sum_{i=1}^k c_1(\alpha_i))^2,
\]
holds.
\end{lm}
\begin{proof}
By Lemma \ref{lm:*id}, one has the chain of identities:
\begin{multline}
ch_2((\alpha_{1}-1)*\dots*(\alpha_{k}-1))=ch_2(\sum_{q=1}^k \sum_{1\leq i_1<\dots< i_q\leq k} (\alpha_{i_1}-1)\cdots (\alpha_{i_q}-1))=\\
\frac{1}{2}\sum_{i=1}^k (c_1(\alpha_i))^2+\sum_{i<j} c_1(\alpha_i)c_1(\alpha_j)=\frac{1}{2}(\sum_{i=1}^k c_1(\alpha_i))^2.
\end{multline}
\end{proof}
\begin{thm}\label{thm:comp}
$M^6$ is a TNS-manifold iff
\[
H^4(M^6;\R)=\R_{\geq 0}\biggl\la \biggl(\sum_{i=1}^m a_i x_i\biggr)^2\bigg|\ a_i=-1,0,1;\ i=1,\dots,m\biggr\ra.
\]
\end{thm}
\begin{proof}
Due to Corollary \ref{cor:powers} it is enough to prove that $ch_2(S(M^6))=\R\la (\sum_{i=1}^m a_i x_i)^2|\ a_i=-1,0,1,i=1,\dots,m\ra$. Corollary \ref{cor:ratprod} and Proposition \ref{pr:scone} imply that\\
$S(M^6)=\R\la t_{i_1}*\dots* t_{i_k}|\ 1\leq i_1<\dots<i_k\leq m,\ t_{i}=\theta_{i}^{\pm 1}-1, \pm (\theta_i-1)^3\ra$. Notice that for any $x\in K(M^6)$, $ch_2(x* (\pm(\theta_i-1)^3))=ch_2(x)$. It remains to use Lemma \ref{lm:chid}.
\end{proof}

\begin{rem}
Let $v$ be a vertex of the polytope $P^3$. W.l.g. assume that $x_1,x_2,x_3$ correspond to the facets of $P^3$ meeting at $v$. The relations in $H^*(M^6;\R)$ (see Theorem \ref{thm:coh}) imply that the right-hand side of the relation in Theorem \ref{thm:comp} coincides identically with
\[
\R_{\geq 0}\biggl\la \biggl(\sum_{i=1}^m a_i x_i\biggr)^2\bigg|\ a_i=-1,0,1;\ i=4,\dots,m\biggr\ra.
\]
This observation gives a further simplification of the TNS-criterion. Theorem \ref{thm:comp} may be deduced also from Proposition \ref{pr:poly}.
\end{rem}

\begin{ex}\label{ex:3-belt}
One can see that the cones from Theorem \ref{thm:comp} corresponding to toric manifolds with combinatorially equivalent moment polytopes are different. Let $P_1^3,P^3_2\subset \R^3$ be the convex polytopes of the same combinatorial type as shown in Fig. \ref{fig:pic} (namely, the connected sum of two cubes along vertices) with normal vectors
\[
\begin{pmatrix}
1 & 0 & 0 & 2 & 2 & 1 & 0 & 0 & -1\\
0 & 1 & 0 & 2 & 1 & 1 & 0 & -1 & -1\\
0 & 0 & 1 & 1 & 1 & 1 & -1 & -1 & -1
\end{pmatrix},
\begin{pmatrix}
1 & 0 & -1 & 1 & 1 & 0 & 1 & 1 & 0\\
0 & 1 & -1 & 1 & 0 & 0 & 1 & 0 & 0\\
0 & 0 & 0 & 1 & 1 & 1 & -1 & -1 & -1
\end{pmatrix},
\]
respectively. (Indices of the normal vectors of these polytopes are shown in Fig. \ref{fig:pic}.) Both of these polytopes are obtained from the corresponding triangular prisms by truncation of two vertices and two edges. One can easily check that $P_1^3,P^3_2$ are Delzant polytopes. Consider the (smooth projective) toric manifolds $M_1^6,M^6_2$ of complex dimension $3$ corresponding to the polytopes $P_1^3,P^3_2$, resp. Theorem \ref{thm:coh} implies that $x_8x_9,x_7x_9,x_7 x_8,x_6^2,x_5x_6,x_5^2$ and $x_9^2,x_8x_9,x_8^2,x_6^2,x_5x_6,x_3^2$ are the bases of $H^4(M_1;\Q)$ and $H^4(M_2;\Q)$, resp. One can calculate the cones from Theorem \ref{thm:comp} using software programs (e.g. Sage, Singular). Namely, the cones $S(M_1),S(M_2)$ corresponding to $M_1,M_2$ have extreme rays (w.r.t. the above bases)
\[
\begin{pmatrix}
0 & 0 & 0 & 0 & 0 & 0 & 0 & 0 & 1 & -1\\
0 & 0 & 0 & 0 & 0 & 0 & 1 & -1 & 0 & 0\\
0 & 0 & 0 & 0 & 1 & -1 & 0 & 0 & 0 & 0\\
1 & 1 & 1 & 0 & 0 & 0 & 0 & 0 & 0 & 0\\
2 & 0 & -2 & 0 & 0 & 0 & 0 & 0 & 0 & 0\\
1 & 0 & 1 & 1 & 0 & 0 & 0 & 0 & 0 & 0
\end{pmatrix},
\begin{pmatrix}
0 & 0 & 1 & 0 & 0 & 1 & 1\\
0 & 0 & 2 & 0 & 0 & -2 & 0\\
0 & 0 & 1 & 0 & 1 & 1 & 0\\
0 & 1 & 0 & 1 & 0 & 0 & 0\\
0 & -2 & 0 & 2 & 0 & 0 & 0\\
1 & 0 & 0 & 0 & 0 & 0 & 0
\end{pmatrix},
\]
and have $4$ and $7$ facets, resp. In particular, $M_1,M_2$ are not TNS-manifolds.
\end{ex}

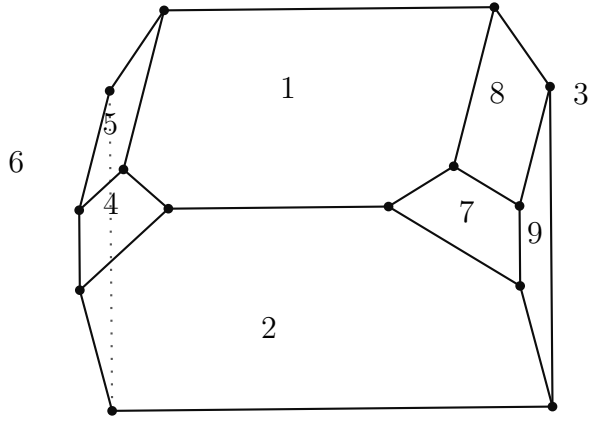
\begin{figure}
\centering
\caption{The combinatorial polytope with facet numbers.\label{fig:pic}}
\begin{tikzpicture}%
	[x={(0.189013cm, -0.711049cm)},
	y={(0.178207cm, 0.703079cm)},
	z={(0.965669cm, 0.009427cm)},
	scale=0.7500000,
	back/.style={loosely dotted, thick},
	edge/.style={color=white!5!black, thick},
	facet/.style={fill=white, fill opacity=0.300000},
	facetlast/.style={fill=white!50!black,fill opacity=0.300000},
	vertex/.style={inner sep=1pt,circle,draw=white!5!black,fill=white!5!black,thick,anchor=base}]	
%
%
\coordinate (4.00000, -2.00000, -4.00000) at (4.00000, -2.00000, -4.00000);
\coordinate (4.00000, -2.00000, 4.00000) at (4.00000, -2.00000, 4.00000);
\coordinate (0.00000, 2.00000, 4.00000) at (0.00000, 2.00000, 4.00000);
\coordinate (0.00000, 2.00000, -4.00000) at (0.00000, 2.00000, -4.00000);
\coordinate (0.00000, -1.00000, 4.00000) at (0.00000, -1.00000, 4.00000);
\coordinate (-1.00000, 3.00000, 3.00000) at (-1.00000, 3.00000, 3.00000);
\coordinate (-1.00000, 3.00000, -3.00000) at (-1.00000, 3.00000, -3.00000);
\coordinate (1.00000, -2.00000, 4.00000) at (1.00000, -2.00000, 4.00000);
\coordinate (-1.00000, -2.00000, 2.00000) at (-1.00000, -2.00000, 2.00000);
\coordinate (-1.00000, -1.00000, 3.00000) at (-1.00000, -1.00000, 3.00000);
\coordinate (-1.00000, -2.00000, -2.00000) at (-1.00000, -2.00000, -2.00000);
\coordinate (-1.00000, -1.00000, -3.00000) at (-1.00000, -1.00000, -3.00000);
\coordinate (1.00000, -2.00000, -4.00000) at (1.00000, -2.00000, -4.00000);
\coordinate (0.00000, -1.00000, -4.00000) at (0.00000, -1.00000, -4.00000);
\draw[edge,back] (4.00000, -2.00000, -4.00000) -- (0.00000, 2.00000, -4.00000);
\fill[facet] (1.00000, -2.00000, 4.00000) -- (4.00000, -2.00000, 4.00000) -- (0.00000, 2.00000, 4.00000) -- (0.00000, -1.00000, 4.00000) -- cycle {};
\fill[facet] (0.00000, -1.00000, -4.00000) -- (-1.00000, -1.00000, -3.00000) -- (-1.00000, -2.00000, -2.00000) -- (1.00000, -2.00000, -4.00000) -- cycle {};
\fill[facet] (1.00000, -2.00000, -4.00000) -- (4.00000, -2.00000, -4.00000) -- (4.00000, -2.00000, 4.00000) -- (1.00000, -2.00000, 4.00000) -- (-1.00000, -2.00000, 2.00000) -- (-1.00000, -2.00000, -2.00000) -- cycle {};
\fill[facet] (-1.00000, -1.00000, 3.00000) -- (0.00000, -1.00000, 4.00000) -- (0.00000, 2.00000, 4.00000) -- (-1.00000, 3.00000, 3.00000) -- cycle {};
\fill[facet] (-1.00000, -1.00000, -3.00000) -- (-1.00000, 3.00000, -3.00000) -- (-1.00000, 3.00000, 3.00000) -- (-1.00000, -1.00000, 3.00000) -- (-1.00000, -2.00000, 2.00000) -- (-1.00000, -2.00000, -2.00000) -- cycle {};
\fill[facet] (0.00000, -1.00000, -4.00000) -- (0.00000, 2.00000, -4.00000) -- (-1.00000, 3.00000, -3.00000) -- (-1.00000, -1.00000, -3.00000) -- cycle {};
\fill[facet] (-1.00000, -1.00000, 3.00000) -- (0.00000, -1.00000, 4.00000) -- (1.00000, -2.00000, 4.00000) -- (-1.00000, -2.00000, 2.00000) -- cycle {};
\draw[edge] (4.00000, -2.00000, -4.00000) -- (4.00000, -2.00000, 4.00000);
\draw[edge] (4.00000, -2.00000, -4.00000) -- (1.00000, -2.00000, -4.00000);
\draw[edge] (4.00000, -2.00000, 4.00000) -- (0.00000, 2.00000, 4.00000);
\draw[edge] (4.00000, -2.00000, 4.00000) -- (1.00000, -2.00000, 4.00000);
\draw[edge] (0.00000, 2.00000, 4.00000) -- (0.00000, -1.00000, 4.00000);
\draw[edge] (0.00000, 2.00000, 4.00000) -- (-1.00000, 3.00000, 3.00000);
\draw[edge] (0.00000, 2.00000, -4.00000) -- (-1.00000, 3.00000, -3.00000);
\draw[edge] (0.00000, 2.00000, -4.00000) -- (0.00000, -1.00000, -4.00000);
\draw[edge] (0.00000, -1.00000, 4.00000) -- (1.00000, -2.00000, 4.00000);
\draw[edge] (0.00000, -1.00000, 4.00000) -- (-1.00000, -1.00000, 3.00000);
\draw[edge] (-1.00000, 3.00000, 3.00000) -- (-1.00000, 3.00000, -3.00000);
\draw[edge] (-1.00000, 3.00000, 3.00000) -- (-1.00000, -1.00000, 3.00000);
\draw[edge] (-1.00000, 3.00000, -3.00000) -- (-1.00000, -1.00000, -3.00000);
\draw[edge] (1.00000, -2.00000, 4.00000) -- (-1.00000, -2.00000, 2.00000);
\draw[edge] (-1.00000, -2.00000, 2.00000) -- (-1.00000, -1.00000, 3.00000);
\draw[edge] (-1.00000, -2.00000, 2.00000) -- (-1.00000, -2.00000, -2.00000);
\draw[edge] (-1.00000, -2.00000, -2.00000) -- (-1.00000, -1.00000, -3.00000);
\draw[edge] (-1.00000, -2.00000, -2.00000) -- (1.00000, -2.00000, -4.00000);
\draw[edge] (-1.00000, -1.00000, -3.00000) -- (0.00000, -1.00000, -4.00000);
\draw[edge] (1.00000, -2.00000, -4.00000) -- (0.00000, -1.00000, -4.00000);
\node[vertex] at (4.00000, -2.00000, -4.00000)     {};
\node[vertex] at (4.00000, -2.00000, 4.00000)     {};
\node[vertex] at (0.00000, 2.00000, 4.00000)     {};
\node[vertex] at (0.00000, 2.00000, -4.00000)     {};
\node[vertex] at (0.00000, -1.00000, 4.00000)     {};
\node[vertex] at (-1.00000, 3.00000, 3.00000)     {};
\node[vertex] at (-1.00000, 3.00000, -3.00000)     {};
\node[vertex] at (1.00000, -2.00000, 4.00000)     {};
\node[vertex] at (-1.00000, -2.00000, 2.00000)     {};
\node[vertex] at (-1.00000, -1.00000, 3.00000)     {};
\node[vertex] at (-1.00000, -2.00000, -2.00000)     {};
\node[vertex] at (-1.00000, -1.00000, -3.00000)     {};
\node[vertex] at (1.00000, -2.00000, -4.00000)     {};
\node[vertex] at (0.00000, -1.00000, -4.00000)     {};

\draw (-2,0) node {$1$};
\draw (0,-4) node {$2$};
\draw (12,14) node {$3$};
\draw (-9,-10) node {$4$};
\draw (-10,-9) node {$5$};
\draw (-14,-14) node {$6$};
\draw (8,7) node {$7$};
\draw (8,10) node {$8$};
\draw (11.5,10) node {$9$};
\end{tikzpicture} 
\end{figure}

\section{Operations on quasitoric TNS-manifolds}

The current Section contains some results on equivariant connected sum and TNS-property (Subsection \ref{ssec:sum}). The blow-up along a complex codimension $2$ invariant submanifold of a TNS toric manifold is generalised for quasitoric manifolds in Subsection \ref{ssec:blow}.

\subsection{Equivariant connected sum and quasitoric TNS-manifolds}\label{ssec:sum}

We begin this Subsection with a recap on signs of fixed points of quasitoric manifolds. (See \cite[Section 7.3]{bu-pa-15}.) Let $M^{2n}$ be a $2n$-dimensional quasitoric manifold, and let $x\in M^{2n}$ be a fixed point under the natural action of the torus $T^n$. The fixed point $x$ is an intersection of pairwise different invariant submanifolds $M_{j_1},\dots,M_{j_n}$ of codimension $2$ in $M^{2n}$. The sign $\sigma(x)$ of $M^{2n}$ at the fixed point $x$ is defined as $1$, if the orientation of the real space $T_{x} M^{2n}\oplus\R^{2(m-n)}$ coincides with the orientation of the (realification) of the vector space $(\theta_{j_{1}}\oplus\dots\oplus\theta_{j_{n}})_{x}$ determined by the orientation of the complex line bundles $\theta_{j_k},\ k=1,\dots,n$, and is equal to $-1$, otherwise. One also has the formula
\begin{equation}\label{eq:sign}
\sigma(x)=\det (\lambda_{j_1},\dots,\lambda_{j_n}),
\end{equation}
provided that the respective inward-pointing normal vectors of the facets $F_{j_1},\dots,F_{j_n}$ of $P^n$ form a positive basis of $\R^n$.

For any two stably complex manifolds $M_1, M_2$ of dimension $2n$ there exists a natural stably complex structure on the connected sum $M_1\sharp_{x_1,x_2} M_2,\ x_i\in M_i,\ i=1,2$ (see \cite[Construction 9.1.11]{bu-pa-15}). Choose a fixed point $x_i\in M_{i}$ of the quasitoric manifold $M_i^{2n}$ under the natural torus action, $i=1,2$. Then the respective connected sum is called an \emph{equivariant connected sum} $M_1\widetilde{\sharp}_{x_1,x_2} M_2$ of the quasitoric manifolds $M_1\widetilde{\sharp}_{x_1,x_2} M_2$ along the fixed points $x_1,x_2$. It admits the action of the torus $(\C^{\times})^{n}$ being a quasitoric manifold. The moment polytope of $M_1\widetilde{\sharp}_{x_1,x_2} M_2$ is equal to the connected sum of the moment polytopes corresponding to $M_1$ and $M_2$, at the vertices corresponding to the fixed points $x_1,x_2$. The characteristic matrix of $M_1\widetilde{\sharp}_{x_1,x_2} M_2$ is given explicitly in \cite{bu-pa-15}. In order to define a (an invariant) stably complex structure on $M_1\widetilde{\sharp}_{x_1,x_2} M_2$, one needs to endow this manifold with an orientation (for example, given by the respective characteristic matrix). The manifold $M_1\widetilde{\sharp}_{x_1,x_2} M_2$ is oriented diffeomorphic either to $M_1\sharp_{x_1,x_2} M_2$ or to $M_1\sharp_{x_1,x_2} \overline{M_2}$. The restriction of the orientation on $M_1\widetilde{\sharp}_{x_1,x_2} M_2$ to $M_{1}$ is equal to the orientation on $M_1$. Restricted to $M_2$, the orientation is equal or opposite to the orientation on $M_2$. In the first case, the introduced orientation on $M_1\widetilde{\sharp}_{x_1,x_2} M_2$ is called \emph{compatible} with the orientations on $M_{i},\ i=1,2$.

\begin{pr}\cite[Lemma 9.1.12]{bu-pa-15}\label{pr:compatiblesum}
The equivariant connected sum $M_1\widetilde{\sharp}_{x_1,x_2} M_2$ of omnioriented quasitoric manifolds admits an orientation compatible with the orientations of $M_{i},i=1,2$, iff $\sigma(x_1)=-\sigma(x_2)$.
\end{pr}

\begin{pr}\label{pr:connsum}
Let $M_1^n,\ M_2^n$ be closed oriented manifolds of dimension $n$. Then the isomorphism of algebras
\[
\tilde{H}^*(M_1\sharp X;\R)\simeq (\tilde{H}^*(M_1;\R)\oplus\tilde{H}^*(X;\R))/I,
\]
holds, where $X=M_{2},\overline{M_2}$, $I=(D[*_{M_1}]-D[*_{X}])$, and $D$ denotes the Poincar\'e duality operator. The multiplication between direct summands above is trivial in all dimensions.
\end{pr}
\begin{proof}
The desired map is induced by the contraction map of the connected sum of oriented manifolds $M_1\sharp X\to M_1\vee X$.
\end{proof}

\begin{pr}\label{pr:eqsum}
Consider the equivariant connected sum $M^{2n}=M_{1}^{2n}\sharp M_{2}^{2n}$ of quasitoric manifolds $M_{1}^{2n}, M_{2}^{2n}$ with the orientation compatible with the orientations on the manifolds $M_{1}^{2n}, M_{2}^{2n}$.

a) Let $n$ be odd. Then $M^{2n}$ is a TNS-manifold iff $M^{2n}_{i}$ is a TNS-manifold, $i=1,2$;

b) Let $n$ be even. Then $M^{2n}$ is a TNS-manifold iff the homogeneous forms $Q_a,\ a\in H^{2(n-2k)}(M^{2n}_i;\Z)$ of degree $0<k<[n/2]$ are admissible for $M_i$, $i=1,2$, and the sum of the forms $Q+Q'$ is admissible, where $Q:H^2(M^{2n}_{1};\Z)\to \Z,\ Q':H^2(M^{2n}_{2};\Z)\to \Z$ are the homogeneous forms of the top degree corresponding to the units $1_{M^{2n}_{1}}\in H^0(M^{2n}_{1};\Z),\ 1_{M^{2n}_{2}}\in H^0(M^{2n}_{2};\Z)$ in sense of Proposition \ref{pr:corr}, respectively.
\end{pr}
\begin{proof}
Follows from Corollary \ref{cor:powers} and Proposition \ref{pr:connsum}.
\end{proof}

Now let $n=2$. 

\begin{lm}\label{lm:signa}
Let $(a_i,b_i)$ be the index of the intersection form of a closed oriented simply-connected $4$-manifold $M_i^{4},i=1,2$. Then the indices of the intersection forms of the manifolds $M_1\sharp M_2, M_1\sharp \overline{M_2}$, are equal to $(a_1+a_2,b_1+b_2),(a_1+b_2,b_1+a_2)$, resp.
\end{lm}
\begin{proof}
Due to Proposition \ref{pr:connsum}, the intersection form of the manifold $M_1\sharp X$ is equal to the direct sum of the respective intersection forms of $M_{1}, X$, where $X=M_2, \overline{M_2}$. It remains to notice that the intersection form of $\overline{M_2}$ is equal to the minus intersection form of $M_2$.
\end{proof}

The indices of the intersection forms of quasitoric manifolds $\C P^{2}, \overline{\C P^{2}}$ are equal to $(1,0),(0,1)$, resp. The signs of all fixed points of $\C P^{2}$ ($\overline{\C P^{2}}$, resp.) are equal to $1$ ($-1$, resp.). Define the quasitoric manifold $B^4:=\C P^{2}\widetilde{\sharp} \overline{\C P^{2}}$. It does not depend on the choice of the fixed points.

Let's explicitly check the following (without using Theorem \ref{thm:crit}).
\begin{pr}
$B^4$ is a TNS-manifold.
\end{pr}
\begin{proof}
By Theorem \ref{thm:kdesc}, it is enough to prove that for any linear bundle $\xi\to B^4$ the stably inverse to $\xi\oplus\bar\xi$ is totally split. Proposition \ref{pr:connsum} implies that
\[
H^*(B^4;\Z)=\Z[x,y]/(x^3,y^3,xy,x^2+y^2).
\]
Let $t:=c_1(\xi)$. One has $t^2=ax^2$ for the generator $x^2\in H^4(B^4;\Z)$ and some $a\in\Z$. Let $\eta_1,\eta_2$ be the pullbacks of the tautological line bundles from $\C P^2,\overline{\C P^2}$ to $B^4$, resp. One has $c_1^2(\bar\eta_1)=x^2,c_1^2(\bar\eta_2)=-x^2$. Consider three cases.

$1)$ If $a=0$, then $\xi\oplus\bar\xi$ is trivial, and there is nothing to prove.

$2)$ If $a<0$ then $ch([\xi\oplus\bar\xi\oplus |a|(\eta_2\oplus\bar\eta_2)])=2+2|a|$.

$3)$ If $a>0$, then $ch([\xi\oplus\bar\xi\oplus a(\eta_1\oplus\bar\eta_1)])=2+2a$.

The statement now follows from Theorem \ref{thm:lat}.
\end{proof}

The manifold $B^4$ has exactly $2$ fixed points of each sign. Choose fixed points $x_{+},x_{-}\in B^4$ s.t. $\sigma(x_+)=1,\sigma(x_-)=-1$.

\begin{pr}\label{pr:eqTNS}
Let $x\in M^4$ be a fixed point of an arbitrary quasitoric manifold $M^4$. Then the equivariant connected sum $M^4\widetilde{\sharp}_{x,x_+} B^4$ ($M^4\widetilde{\sharp}_{x,x_-} B^4$, resp.) is a TNS-manifold, if $\sigma(x)=-1$ ($\sigma(x)=1$, resp.). The respective orientation is compatible with the orientations on $M^4,B^4$.
\end{pr}

\subsection{Equivariant blow-ups and quasitoric TNS-manifolds}\label{ssec:blow}

Let $M^{2n}=M(P,\Lambda)$ be the model of the quasitoric manifold given by a characteristic pair $(P,\Lambda)$. For any vertex $v=F_{i_1}\cap\dots\cap F_{i_n}\in P$ the column vectors $\lambda_{i_1},\dots, \lambda_{i_n}$ of the characteristic $(n\times m)$-matrix $\Lambda=(\lambda_i^j)$ generate the cone $\sigma_v=\cone(\lambda_{i_{1}},\dots,\lambda_{i_{n}})\subset \Z^n\otimes\R$.

Consider two characteristic pairs $(P,\Lambda),(P',\Lambda')$. Let $N,N'$ be the lattices generated by the column vectors of $\Lambda,\Lambda'$, respectively. A weakly equivariant morphism $\phi:M(P,\Lambda)\to M(P',\Lambda')$ between two quasitoric manifolds is by definition given by the lattice map $\psi: N\to N'$ s.t. for any vertex $v\in P$ there exists a vertex $v'\in P'$ with $\psi(\sigma_v)\subset \sigma_{v'}$.

Let $Z\subset M$ be a characteristic submanifold of $M$ of codimension $2k$ corresponding to the face $G=F_{i_1}\cap\dots\cap F_{i_k}\subset P$. Consider the Delzant polytope $\widetilde{P}=cut_G P$ which is the truncation of the face $G$ of $P$, and a characteristic matrix $\widetilde{\Lambda}$ obtained from $\Lambda$ by adding the vector $\lambda=\lambda_{i_1}+\dots+\lambda_{i_k}$ corresponding to the truncation facet of $\widetilde{P}$. The identity map on the lattices generated by the columns of $\Lambda, \Lambda'$ gives rise to a weakly equivariant morphism $\pi: M(\widetilde{P},\widetilde{\Lambda})\to M(P,\Lambda)$.

\begin{definition}
The weakly equivariant morphism $\pi: Bl_{Z^{2(n-k)}}M^{2n}=M(\widetilde{P},\widetilde{\Lambda})\to M(P,\Lambda)$ of quasitoric manifolds is called an equivariant blow-up of a quasitoric manifold $M$ at a characteristic submanifold $Z\subset M$.
\end{definition}

\begin{pr}\label{pr:blowup}
Let $Z\subset M$ be a characteristic submanifold of codimension $4$ in a quasitoric TNS-manifold $M$. Then the equivariant blow-up $Bl_{Z} M^{2n}$ is a TNS-manifold.
\end{pr}
\begin{proof}
Let $\widetilde{\theta},\widetilde{\theta}_1,\dots,\widetilde{\theta}_m\to M^{2n}$ be the complex linear bundles corresponding to the truncation facet and all other facets of $\cut_{G} P$ in the same order as in $P$, respectively. Due to Theorem \ref{thm:equiv} $(ii)$, it is enough to check that any of the linear bundles $\widetilde{\theta},\widetilde{\theta}_1,\dots,\widetilde{\theta}_m$ has a totally split stably inverse. One has $\widetilde{\theta}_i=\pi^*\theta_{i}$ for any $i=1,\dots,m$. Hence, $\widetilde{\theta}_{i}$ has a totally split stably inverse for any $i=1,\dots,m$ (see Proposition \ref{pr:vx}).

It remains to deal with the stably inverse to $\widetilde{\theta}$. Let $y_{i}:=c_{1}(\widetilde{\theta}_i), y:=c_{1}(\widetilde{\theta})$. Consider the face $G=F_{i_1}\cap F_{i_2}\subset P^n$. W.l.g. we may assume that $\lambda_{i_1}=(1,0,0,\dots,0)^T,\lambda_{i_1}=(0,1,0,\dots,0)^T$ by using the necessary coordinate change. One has relations $y_{i_1}y_{i_2}=0, y_{i_1}=-y+a,y_{i_2}=-y+b$ in the cohomology ring $H^{*}(Bl_{Z} M;\Z)$, where $a=-\sum_{i=3}^m \lambda_{i_1}^i y_i,b=-\sum_{i=3}^m \lambda_{i_2}^i y_i$. Indeed, the first relation follows from the void intersection of the facets $F_{i_1},F_{i_2}$ of $\widetilde P$, and the other relations are corollaries of the linear relations in $H^{*}(Bl_{Z} M;\Z)$. Let $\alpha,\beta\to Bl_{Z} M$ be the complex linear bundles s.t. $c_1(\alpha)=a,c_1(\beta)=b$. Observe that $c(\widetilde{\theta}(\alpha\oplus\beta))=(1+y+a)(1+y+b)=1+2y+a+b=c(\widetilde{\theta}^2\alpha\beta)$, hence $\widetilde{\theta}(\alpha\oplus\beta)\simeq \widetilde{\theta}^2\alpha\beta\oplus\underline{\C}$. We conclude that $\widetilde{\theta}(\alpha\oplus\beta)\oplus\overline{\widetilde{\theta}(\alpha\oplus\beta)}\simeq \underline{\C}^4$, i.e. the stably inverse to $\widetilde{\theta}\alpha$ totally splits. By Proposition \ref{pr:vx}, the bundles $\alpha,\beta$ have totally split stably inverses. Applying Proposition \ref{pr:vx} again to the tensor product of $\widetilde{\theta}\alpha$ and $\overline{\alpha}$ we obtain the required.
\end{proof}

\section{Smooth projective toric TNS-manifolds}

The main result of this Section is Theorem \ref{thm:miss} describing the properties of the moment polytope for a smooth projective toric TNS-manifold.

Let $M^{2n}$ be a non-singular projective toric variety of complex dimension $n>2$. Let $P^n\subset \R^{n}$ be the respective moment polytope of $M^{2n}$. Let $F_{i}$, be the facets of $P^n$ with the respective normal vectors $\lambda_{i}=(\lambda_i^1,\lambda_i^2,\dots,\lambda_i^n)^{T}\in\R^n$, $i=1,\dots,m$. Let $K$ be the simplicial complex corresponding to $P^n$ (i.e. the face lattice of the simplicial sphere $P^*\subset(\R^n)^{*}$). Consider a minimal missing face of $K$ on $k$ vertices. W.l.g. we may assume that vertices of that missing face correspond to the facets $F_1,\dots,F_k\subset P$: $F_1\cap\dots\cap\wideparen{F_{i}}\cap\dots\cap F_k\neq\varnothing, 1\leq i\leq k,\ F_1\cap\dots\cap F_k=\varnothing$.
\begin{pr}\label{pr:2cases}
Let $k>2$. Then the matrix $(\lambda_1,\dots,\lambda_k)$ is $GL_n(\Z)$-equivalent either to\\
$\begin{pmatrix}
1 & 0 & \dots & 0 & 0\\
0 & 1 & \dots & 0 & 0\\
\dots & \dots & \dots & \dots & \dots\\
0 & 0 & \dots & 1 & 0\\
0 & 0 & \dots & 0 & 1\\
0 & 0 & \dots & 0 & 0\\
\dots & \dots & \dots & \dots & \dots\\
0 & 0 & \dots & 0 & 0
\end{pmatrix}$
or to
$\begin{pmatrix}
1 & 0 & \dots & 0 & -1\\
0 & 1 & \dots & 0 & -1\\
\dots & \dots & \dots & \dots & \dots\\
0 & 0 & \dots & 1 & -1\\
0 & 0 & \dots & 0 & 0\\
0 & 0 & \dots & 0 & 0\\
\dots & \dots & \dots & \dots & \dots\\
0 & 0 & \dots & 0 & 0
\end{pmatrix}$.
\end{pr}
\begin{proof}
Let $i>k$ be a facet index s.t. $F_1\cap\dots\cap F_{k-1}\cap F_i$ contains a vertex of $P$. Consider the standard unitary basis $e_1,\dots,e_n\in\R^n,\ e_i^j=\delta_{i}^j$. Since $P^n$ is a Delzant polytope, there is a $GL_n(\Z)$-transform taking $\lambda_1,\dots,\lambda_{k-1}$ to the vectors $e_1,\dots,e_{k-1}$, resp. Clearly, the matrix $(\lambda_1,\dots,\lambda_k)$ has rank $k-1$ or $k$. In the first case, one has $\lambda_k=(a_{1},\dots,a_{k-1},0,\dots,0)^T$ for some $a_{i}\in \Z,\ i=1,\dots,k-1$. For any $i=1,\dots,k-1$ there is $j>k$, s.t. $F_1\cap\dots \cap\wideparen{F_i}\cap\dots\cap F_k\cap F_j$ has a vertex. Then the Delzant condition implies $a_{i} c_i=1$ for some $c_i\in\Z$. Hence, $a_{i}=\pm 1$ for $i=1,\dots,k-1$. Since $\eta_{k}$ is an inward-pointing normal vector to $F_{k}$, one has $a_{1}=\dots=a_{k}=-1$.

Now suppose that $\rk (\lambda_1,\dots,\lambda_k)=k$. Applying the corresponding $GL_n(\Z)$-transform to $P^n$ one obtains $\eta_{k}=c e_k$ for some $c\in\Z$. The Delzant condition implies $c d=-1$ for some $d\in\Z$. Hence $c=\pm 1$. Applying the corresponding $GL_n(\Z)$-transform one reduces the matrix to the desired form.
\end{proof}

Let $n=k=3$. Clearly, $\lb 1,\dots ,m \rb=\lb 1,2,3\rb\sqcup S_1\sqcup S_2$, where $S_1,S_2$ are s.t. $\bigcup_{i\in S_1} F_i$ and $\bigcup_{j\in S_2} F_j$ have empty intersection.

Consider the case $(\lambda_1,\lambda_2,\lambda_3)=
\begin{pmatrix}
1 & 0 & 0\\
0 & 1 & 0\\
0 & 0 & 1
\end{pmatrix}$.

\begin{pr}\label{pr:posi}
There exist $i\in \lb 1,2\rb$ s.t. for any $p\in S_{i},\ k=1,2,3$ one has
\[
\lambda_{p}^k\geq 0.
\]
\end{pr}
\begin{proof}
Denote by $x\in\R^3$ the common vertex belonging to planes containing the facets $F_{1},F_{2},F_{3}$. Consider the minimal convex polyhedral cone $\sigma\subset\R^3$ containing $P^3$ with origin in $x$. Choose indices $\lb i,j\rb=\lb 1,2\rb$ s.t. $\bigcup_{p\in S_{i}} F_{p}$ is the closest part to $x$. (More precisely, there exists a dividing plane $\lb H=c\rb$ s.t. $\bigcup_{p\in S_{i}} F_{p}\cap \lb H<c\rb=\bigcup_{p\in S_{i}} F_{p}$ and $\bigcup_{p\in S_{j}} F_{p}\cap \lb H<c\rb=\varnothing$.) For any facet $F_{p},p\in S_{i}$, the corresponding plane intersects the interiors of all $3$ facets of $\sigma$. Hence, the normal vector $\eta_{p}$ belongs to the cone generated by $(1,0,0)^{T},(0,1,0)^{T},(0,0,1)^{T}$, and has only non-negative coordinates.
\end{proof}

Denote by $S_{+},S_{\pm}$ the sets $S_{i},S_{j}$ from the Proposition \ref{pr:posi}, resp. (By abuse of the notation, the values $i,j$ are now undefined.) The characteristic matrix of $M^{6}$ has form $(Id_{3}|A|B)$, where $A=(a_{i}^{j}),\ a_{i}^{j}\geq 0$, $B=(b_{i}^{j})$. By the definition, $M^6$ is the equivariant connected sum of quasitoric manifolds $M_{+}^{6},M_{\pm}^{6}$, where $M_{+}^{6}=M(P_{+},(Id_3|A)),\ M_{\pm}^{6}=M(P_{\pm},(Id_3|B))$, and $P_{+},P_{\pm}\subset\R^3$ are the respective combinatorial convex polytopes. Let $F_{+,i}$ be the facets of $P_{+}$ corresponding to $F_{i}\in P^3$, $i\in \lb 1,2,3\rb\cup S_{+}$.

Consider the characteristic submanifold $M^{4}_{+,i}$ of $M^{6}_{+}$ corresponding to the facet $F_{+,i},\ i=1,2,3$. The corresponding moment polygone and characteristic matrix are denoted by $P^{4}_{+,i}, \Lambda_{+,i}$, resp. The matrix $\Lambda_{+,i}$ is equal to the submatrix of $\Lambda$ consisting of the rows complement to the $i$-th and columns with indices $(\lb 1,2,3\rb\setminus\lb i\rb)\cup \lb k\in S_{+}|\ F_{i}\cap F_{k}\neq\varnothing\rb$ (see \cite{bu-pa-15}). The orientation on $M^6$ is compatible with the orientations on the summands $M^6_{+}$, $M^{6}_{\pm}$ (see Proposition \ref{pr:compatiblesum}).

\begin{pr}\cite[Lemma 7.3.19]{bu-pa-15}\label{pr:torsign}
For any non-singular projective toric variety $X$ and any fixed point $x\in X$ one has $\sigma(x)=1$.
\end{pr}

\begin{lm}\label{lm:sgncomp}
For any index $i=1,2,3$ and a vertex $v\in P^{4}_{+,i}$, one has
\[
\sigma(v)=
\begin{cases}
-1, \mbox{ if $v$ is different from } F_{+,1}\cap F_{+,2} \cap F_{+,3} \mbox{ and incident to it},\\
1, \mbox{otherwise}.
\end{cases}
\]
\end{lm}
\begin{proof}
W.l.g. let $i=1$. Proposition \ref{pr:torsign} and Formula \eqref{eq:sign} imply that for any vertex $v\in P^{4}_{+,1}$ different from $F_{+,1}\cap F_{+,2} \cap F_{+,3}$ and not incident with the latter, one has $\sigma(v)=1$. The characteristic vectors of the edges $F_{1}\cap F_{2}, F_{1}\cap F_{3}$ are outward-pointing, whereas the characteristic vectors of all other edges in $P_{+,1}$ are inward-pointing. This implies the desired identities on the signs of the remaining $3$ vertices of $P_{+,1}$.
\end{proof}

Remind that the index $ind_{\nu} (v)$ of a vertex $v=F_{i_1}\cap\dots\cap F_{i_n}\in P^n$ of the moment polytope w.r.t. the generic vector $\nu\in N\otimes\R^n$ is by definition equal to the number of negative scalar products of $\nu$ with the conjugate basis to $\lambda_{i_1},\dots, \lambda_{i_n}$ (see \cite[Section 7.3]{bu-pa-15}).

\begin{pr}\cite[Theorem 9.4.8]{bu-pa-15}\label{pr:sgnform}
\[
\sigma(M^{2n})=\sum_{v\in \vertex P^n } (-1)^{ind_{\nu}(v)} \sigma(v).
\]
\end{pr}


\begin{lm}\label{lm:indcomp}
Let $i=1,2,3$. Then there exists $\nu\in\R^2$ (in general position) and a vertex $w\in P^{4}_{+,i}$, $w\neq F_{+,1}\cap F_{+,2} \cap F_{+,3}$, s.t. for any vertex $v\in P^{4}_{+,i}$, one has $ind_{\nu} (v)=0$, if $v=F_{+,1}\cap F_{+,2} \cap F_{+,3}$ or $v=w$, and $ind_{\nu} (v)=1$, otherwise.
\end{lm}
\begin{proof}
Let $v_0:=F_{+,1}\cap F_{+,2}\cap F_{+,3}$. The basis to the columns of $(Id_2)$ is conjugate to itself. Moreover, the scalar products of these columns with any vector $\nu\in\R_{> 0}^2$ having positive coordinates are positive. Hence, for any $\nu\in\R^2_{>0}$ one has
$\ind_{\nu} v_0=0$ w.r.t. the characteristic pair $(P_{+,i},(Id_2|A_i))$.

Now we compute the indices of the remaining vertices of $(P_{+,i},(Id_2|A_i))$. Consider the connected sum $R:=P_{+,i}\#_{v_0,u} I^2$ with the rectangle, where $u\in\vertex I^2$. The corresponding characteristic matrix is equal to $(-Id_2|Id_2|A_i)$. The edges and vertices of the polygon $P_{+,i}$ (except $v_0$) may be identified with the corresponding edges and vertices of the polygon $R$. The columns of the matrix $(-Id_3|Id_3|A)$ are normal vectors to the facets of the Delzant polytope combinatorially equivalent to the connected sum $P_{+}\#_{v_0,u} I^3$, where $u\in\vertex I^3$. Hence, the columns of the matrix $(-Id_2|Id_2|A_i)$ are normal vectors of a Delzant polygon.

Pick a vector $\nu\in\R^2_{>0}$ with positive coordinates s.t.
$\nu$ is not orthogonal to neither of the columns of $(-Id_2|Id_2|A_i)$ and s.t. the linear function $H:\ \R^2\to\R$ dual to $\nu$ takes the only maximum at the vertex $v_0$ and takes the only minimum at some vertex $w$ of the convex polytope $R$. The linear function $H$ induces the direction of the edges of $R$ towards its maximum. The bases consisting of the normal vectors to the edges of $R$ meeting at $v$ and of the vectors along these edges are conjugate to each other. Hence, the index $\ind_{\nu} v$ of any vertex $v$ of $R$ is equal to the number of the incoming edges at $v$ due to \cite[p.252, p.378]{bu-pa-15}. Thus we obtain the required formula w.r.t. the characteristic pair $(P_{+,i},(Id_2|A_i))$ for any vertex $v$ different from $v_0$.
\end{proof}


\begin{pr}\label{pr:nontns}
For any $i=1,2,3$, $M^4_{+,i}$ is not a TNS-manifold.
\end{pr}
\begin{proof}
Let $m_i$ be the number of vertices of the polygon $P^4_{+,i}$. It follows from Lemmas \ref{lm:sgncomp}, \ref{lm:indcomp} and Proposition \ref{pr:sgnform} that
$\sigma(M^4_{+,i})=(m_i-2)(-1)^1\cdot 1+(-1)^0\cdot 1+(-1)^0\cdot (-1)=2-m_i=-\dim H^2(M^4_{+,i};\R)$. It remains to use Theorem \ref{thm:crit}.
\end{proof}

Let $l_{i}:=\sum_{j\in S_{+}} a_{i}^{j}x_{j}\in H^2(M^6;\Z)$, $i=1,2,3$.

\begin{pr}
The quadratic forms $Q_{l_{i}}$ are non-trivial and positive-semidefinite for $i=1,2,3$.
\end{pr}
\begin{proof}
Let $\iota:\  M^{6}_{+}\to M^6$, $\iota_{i}:\ M^{4}_{+,i}\to M^6_{+}$ be the natural embeddings of the components $M^{6}_{+}$, $M^{4}_{+,i}$ in the equivariant connected sums $M^6$ and $M^6_{+}$, resp. Denote by $\widetilde{x}$ the expression $\iota^*(x)\in H^{*}(M^6_{+};\ \R)$ for any $x\in H^{*}(M^6;\ \R)$. By Proposition \ref{pr:connsum} and Theorem \ref{thm:coh} the induced maps
$\iota^{*}, \iota^{*}_{i}$ in the second cohomology groups are epimorphisms. The orientation on $M^6_{+}$ coincides with the restriction of the orientation on $M^6$. Hence, by the projection formula one has
\begin{equation}\label{eq:projfla1}
Q_{\widetilde{l_i}}(\widetilde{x})=\la \iota^{*}(l_i\cdot x^2),[M^{6}_{+}]\ra= \la
l_i\cdot x^2,[M^{6}]\ra=Q_{l_i}(x),
\end{equation}
for any $x\in H^{2}(M^6;\ \R)$.

By Theorem \ref{thm:coh}, one has $\widetilde{l_{i}}=-\widetilde{x_i}$ in the ring $H^{*}(M^6_{+};\ \R)$. The submanifold $M^{4}_{+,i}\subset M^6_{+}$ is Poincar\'e dual to the cohomology class $\widetilde{x_i}\in H^{2}(M^6_{+};\ \R)$. Hence, by the projection formula one has
\begin{equation}\label{eq:projfla2}
Q_{1}(\iota_{i}^{*}(\widetilde{x}))=\la \iota_{i}^{*}(\widetilde{x})^2,[M^{4}_{+,i}]\ra= \la
\widetilde{x_i}\cdot \widetilde{x}^2,[M^{6}_{+}]\ra=-\la \widetilde{l_i}\cdot \widetilde{x}^2,[M^{6}_{+}]\ra=-Q_{\widetilde{l_i}}(\widetilde{x}),
\end{equation}
for any $x\in H^{2}(M^6;\ \R)$. Remind that we showed the negative-definiteness of the form $Q_1:\ H^{2}(M^{4}_{+,i};\ \R)\to\R$ in the proof of Proposition \ref{pr:nontns}. Thus, it follows from formulas \eqref{eq:projfla1}, \eqref{eq:projfla2} that the form $Q_{l_i}:\ H^{*}(M^6;\ \R)\to\R$ is non-trivial and positive semidefinite.
\end{proof}

Now suppose that $(\lambda_1,\lambda_2,\lambda_3)=
\begin{pmatrix}
1 & 0 & -1\\
0 & 1 & -1\\
0 & 0 & 0
\end{pmatrix}$.
We return to the denotation $S_1,S_2\subset\lb 1,\dots,m\rb$. Let $l:=\sum_{i\in S_1}\lambda_{3}^{i}x_{i},l':=\sum_{i\in S_2}\lambda_{3}^{i}x_{i}$.

\begin{pr}\label{pr:semidef}
The quadratic form $Q_{l}$ has rank $1$. In particular, it is semidefinite.
\end{pr}
\begin{proof}
Theorem \ref{thm:coh} implies that one has $l=-l'$ in the cohomology ring of $M^{6}$ and $\dim\la x_3,\dots,x_m\ra=m-3$. Hence, it is the only linear dependence between the elements $x_3,\dots,x_m$ in $H^2(M^6;\R)$ up to multiplication by scalars. Then for any $x\in H^2(M^6;\Z)$, $Q_{l}(x)=-Q_{l'}(x)$ holds. For any $i\in S_1, j\in S_2$ one has $F_i\cap F_j=\varnothing$, thus $x_i x_j=0$ (see Theorem \ref{thm:coh}). Now it follows that for any $x\in\R\la x_i| i\in S_1\cup S_2\ra:\ Q_{l}(x)=0$. The Poincar\'e duality implies that $Q_{l}\not\equiv 0$. We conclude that $Q_{l}$ has rank $1$, taking a non-zero value on $x_3$.
\end{proof}

\begin{lm}\label{lm:indmis}
Let $n,k>2$. Suppose that $P^n$ has facets $F_1,\dots,F_k$ corresponding to a minimal missing face of the corresponding face lattice on $P^n$. Then for any $i=1,\dots,k$ the face lattice on the polytope $F_i$ has a minimal missing face $F_1\cap F_i,\dots,\wideparen{F_i\cap F_i},\dots,F_k\cap F_i$ on $k-1$ vertices.
\end{lm}
\begin{proof}
Follows from the fact that the intersection of any two facets of a simple convex $n$-polytope ($n>2$) is either empty or has codimension $2$.
\end{proof}

\begin{lm}\cite{zieg-95}
A simple convex polyhedron $P^3$ is flag iff $P^3$ has no $3$-belts and $P\neq\Delta^3$.
\end{lm}


\begin{thm}\label{thm:miss}
Let $M^{2n}$ be a nonsingular projective toric variety of complex dimension $n\geq 3$. Suppose that $M^{2n}$ is a TNS-manifold. Then the face lattice of the moment polytope $P^{n}$ has no minimal missing faces on $n$ vertices. In particular, the moment polytope $P^3$ of a toric $3$-dimensional TNS-manifold is a flag polytope.
\end{thm}
\begin{proof}
By to Lemma \ref{lm:indmis}, w.l.g. let $n=3$. The polytope $P^3$ of the TNS-manifold $M$ has no triangles according to Proposition \ref{pr:class}. Hence, $P\neq \Delta^3$. Assume that the face lattice of $P^3$ has a minimal non-face on $3$ vertices, i.e. the polytope $P^3$ has a $3$-belt. Denote by $r$ the rank of the matrix with column vectors --- the normals to the facets of the $3$-belt of $P$. By the Proposition \ref{pr:2cases}, one has $r=2$ or $r=3$. In case of $r=3$ $M$ has invariant non-TNS manifolds by Proposition \ref{pr:nontns}. But this contradicts the TNS property of $M$. In case of $r=2$ we obtain the contradiction with the TNS property of $M$ due to Proposition \ref{pr:semidef} and Theorem \ref{thm:lannes}.
\end{proof}

\begin{ex}
Consider the toric manifolds $M_1^6,M^6_2$ from Example \ref{ex:3-belt}. Then the indices of the quadratic forms $Q_{x_{4}+x_{5}+x_{6}}, Q_{2x_{4}+2x_{5}+x_{6}}, Q_{2x_{4}+x_{5}+x_{6}}$ are equal to $(2,0),(2,0),(2,0)$ and $(1,0),(2,1),(2,1)$, resp.
\end{ex}

\section{Concluding remarks}\label{sec:concl}

Proposition \ref{pr:class} and Theorem \ref{thm:miss} hint that there is a possible connection between the combinatorial type of the moment polytope of a smooth projective toric variety and the respective TNS-property.

\begin{conj}\label{conj:comb}
Let $X,Y$ be non-singular projective toric varieties of complex dimension $n$ with combinatorially equivalent moment polytopes. Then $X$ is a TNS-manifold iff $Y$ is a TNS-manifold.
\end{conj}

Proposition \ref{pr:nontns} shows that the analogue of Conjecture \ref{conj:comb} in the category of quasitoric manifolds is, generally speaking, false in any dimension greater than two. Indeed, one has to consider the product of one of the manifolds $M^4$ from Proposition \ref{pr:nontns} with $(\C P^1)^{n-2}$. The obtained manifold then is a quasitoric non-TNS manifold over the cube $I^n$. But $(\C P^1)^n$ is also a (quasi)toric manifold over $I^n$ being a TNS-manifold. Theorem \ref{thm:miss} also allows to pose the following

\begin{conj}
Let $M^{2n}$ be a nonsingular projective toric variety of complex dimension $n$. Then $M^{2n}$ is a TNS-manifold iff the moment polytope $P^{n}$ of $M^{2n}$ is a flag polytope.
\end{conj}

A convex $n$-polytope $P^n\subset\R^n$ is flag iff the corresponding face lattice of the moment polytope $P^n$ has minimal missing faces only on $2$ vertices. So, in order to study the above conjectures in real dimension $8$, one has to find a Delzant $4$-polytope $P^4$ having only facets with no $3$-belts or triangles, and $P^4$ having a minimal missing face on $3$ vertices. Such an example is not known to the author. It is also  plausible to expect the future proofs of the above conjectures to rely on the existence of the complex/algebraic structure on the respective toric variety. In connection with this we mention the different well-known descriptions of the K-theory ring of a toric variety obtained by Pukhlikov and Khovanskii \cite{pu-kh-92}, Morelli \cite{mo-93} and Klyachko. 

The relation between the top-degree form in the TNS-criterion (Theorem \ref{thm:crit}) and the volume polynomial of the multifan of a quasitoric manifold extends to the lower-degree forms. We will use further some auxiliary facts. Recall that for any homogeneous polynomial\\
$V(x_1,\dots,x_m)\in\R[x_1,\dots,x_m]$ of degree $d\geq 0$ the algebra
\[
A(V):=DOp_{\R}(\R^m)/\Ann V
\]
is a Poincar\'e algebra of virtual rank $d$ (see \cite{ti-99}), where $DOp_{\R}(\R^m)$ is the algebra of differential operators in $m$ variables with real constant coefficients and $\Ann V$ is the annihilator ideal of $V$. The natural grading on the algebra $DOp_{\R}(\R^m)$ induces the grading on $A(V)=A_{*}(V)=\bigoplus_{i=0}^{d} A_{i}(V)$. The following theorem was formulated by Pukhlikov and Khovanskii in case of a smooth projective toric variety $M^{2n}$ and its respective fan $\mc{F}$ (in terms of the Chow ring, \cite{pu-kh-93}). It was proved by Timorin \cite{ti-99}. In case of a quasitoric manifold $M^{2n}$ it follows from the results of Ayzenberg and Matsuda \cite{ay-15} about the dual algebra of a multifan, in the particular case of a simplicial complex on a sphere.

\begin{thm}\cite[Theorem 8.2]{ay-15} \label{thm:volcoh}
Let $M^{2n}$ be a quasitoric manifold with the multifan $\mc{F}\subset\R^{n}$ having $m$ rays. Then the isomorphism of algebras
\[
H^{*}(M^{2n};\R)\simeq A_{*}(V_{\mc{F}}),\ a\mapsto D_a,
\]
holds. The canonical pairing $\la a,[M^{2n}]\ra,\ a\in H^{2n}(M^{2n};\R)$, coincides with the evaluating of $D_a V_{\mc{F}}$.
\end{thm}

One has
\begin{lm}\cite[p.19]{ay-15} \label{lm:dval}
Consider a homogeneous polynomial $V(x_1,\dots,x_m)\in\R[x_1,\dots,x_m]$ of degree $d\geq 0$. Let $c_1,\dots,c_m\in\R$. Then for the linear differential operator $D_{\underline{c}}:=c_1\partial_{x_1}+\cdots+c_m\partial_{x_m}$, $\underline{c}=(c_1,\dots,c_m)$, the formula
\[
D_{\underline{c}}^d V=d!V(c_1,\dots,c_m),
\]
holds.
\end{lm}

Define the homogeneous form $Q_{\alpha}: A_{1}(V_{\mc{F}})\to\R$ of degree $k$ by the formula $Q_{\alpha}(x):= \alpha x^k V_{\mc{F}}$, where $k=1,\dots,n$ and $\alpha\in A_{d-k}(V_{\mc{F}})$. Theorem \ref{thm:volcoh} and Lemma \ref{lm:dval} allow us to reformulate the TNS-criterion (Theorem \ref{thm:crit}) in the following way.

\begin{thm}\label{thm:volpol}
Let $M^{2n}$ be a quasitoric manifold with the multifan $\mc{F}\subset\R^{n}$ having $m$ fays. Then the following conditions are equivalent:

$(i)$ $M^{2n}$ is a TNS-manifold;

$(ii)$ One has
\[
\R_{\geq 0}\la x^{k}|\ x\in A_{1}(V_{\mc{F}})\ra=A_{k}(V_{\mc{F}}),
\]
where $k=1,\dots,n$;

$(iii)$ The homogeneous $k$-form $Q_{\alpha}$ is admissible for any $k=1,\dots,n$ and $\alpha\in A_{n-k}(V_{\mc{F}})$;

$(iv)$ For any homogeneous differential operator $D\in DOp_{\R}(\R^m)$, $\deg D=k$, if the polynomial $DV_{\mc{F}}$ is non-zero, then $DV_{\mc{F}}$ takes values of opposite signs, where $k=0,\dots,n-1$. 

\end{thm}

We remark that an analogue of some equivalences in the above Theorem \ref{thm:volpol} takes place for \textit{any} homogeneous polynomial $V(x_1,\dots,x_m)$ of degree $d$ with real coefficients.

\begin{thm}\label{thm:algform}
Let $V(x_1,\dots,x_m)$ be a homogeneous polynomial of degree $d$ with real coefficients. Then the following conditions are equivalent:

$(i)$ One has
\[
\R_{\geq 0}\la x^{k}|\ x\in A_{1}(V)\ra=A_{k}(V),
\]
where $k=1,\dots,d$;

$(ii)$ The homogeneous form $Q_{\alpha}:A_1(V)\to\R$ is admissible for any $k=1,\dots,d$ and $\alpha\in A_{d-k}(V)$;

$(iii)$ For any homogeneous differential operator $D\in DOp_{\R}(\R^m)$, $\deg D=k$, if the polynomial $DV$ is non-zero, then $DV$ takes values of opposite signs, where $k=0,\dots,d-1$.
\end{thm}
\begin{proof}
The equivalence $(ii) \Leftrightarrow (iii)$ follows from Lemma \ref{lm:dval}. The equivalence $(i) \Leftrightarrow (ii)$ is straight-forward to show using the arguments of Subsection \ref{ssec:main}. In order to do that, one has to use two facts about the algebra $A(V)$. First, $A(V)$ is a Poincar\'e algebra (see \cite{ti-99}). Second, for any $k=1,\dots,d$ one has $A_k(V)=\R\la x^k|\ x\in A_1(V)\ra $. The last identity is a consequence of the easily shown formula $(-1)^r r! y_1\cdots y_r=\sum_{I\subseteq\lb 1,\dots,r\rb} (-1)^{|I|}(\sum_{i\in I} y_{i})^r$, taking place in the polynomial algebra $\R[y_1,\dots,y_r],\ r\in\N$.
\end{proof}

The next observation was suggested by A. Ayzenberg.

\begin{cor}\label{cor:alg}
The condition $(iii)$ of Theorem \ref{thm:algform} (and the condition $(iv)$ of Theorem \ref{thm:volpol}) is algorithmically verifiable.
\end{cor}
\begin{proof}
It is clear that this condition could be written as a closed arithmetic formula of first order on the coefficients $d_{i_1,\dots,i_m}$ of the differential operators:
\begin{multline}
\forall k\in\lb 0,\dots,n-1\rb\ \forall \lb d_{i_1,\dots,i_m}|i_1+\cdots+i_m=k\rb:\\
\sum_{i_1+\cdots+i_m=k} d_{i_1,\dots,i_m}\partial_{1}^{i_1}\cdots \partial_{m}^{i_m} V_{\mc{F}}\not \equiv 0 \Rightarrow\\
\bigg(\neg \sum_{i_1+\cdots+i_m=k} d_{i_1,\dots,i_m}\partial_{1}^{i_1}\cdots \partial_{m}^{i_m} V_{\mc{F}}\geq 0 \bigg)
\wedge
\bigg(\neg \sum_{i_1+\cdots+i_m=k} d_{i_1,\dots,i_m}\partial_{1}^{i_1}\cdots \partial_{m}^{i_m} V_{\mc{F}}\leq 0 \bigg).
\end{multline}
Hence, the claim follows now from Tarski algorithm \cite{ta-48}.
\end{proof}

A $17$-th Hilbert's problem-type question for finite-dimensional algebras rises (see \cite[Chapter $7$]{pra-04}).

\begin{prob}
Describe explicitly the family of homogeneous polynomials $V(x_1,\dots,x_m)\in\R[x_1,\dots,x_m]$ of degree $d\geq 2$, $m\geq 1$, satisfying the conditions
of Theorem \ref{thm:algform}. 
\end{prob}

It is natural to conjecture that the condition $(iii)$ of Theorem \ref{thm:algform} is equivalent to the condition on the different signs of non-zero \textit{quadratic} forms $DV_{\mc{F}}$, $D\in DOp_{\R}(\R^m)$. This follows from the following Conjecture about real psd-forms.

\begin{conj}
Consider a homogeneous polynomial $V(x_1,\dots,x_m)\in\R[x_1,\dots,x_m]$ of degree $d\geq 3$, $m\geq 1$. Let $V$ be a psd-form, i.e. for any $x_1,\dots,x_m\in\R$, $V(x_1,\dots,x_m)\geq 0$. Then there exists a homogeneous differential operator $D\in DOp_{\R}(\R^m)$ with constant real coefficients, $0<\deg D<d$, s.t. the polynomial $DV$ is a non-zero psd-form.
\end{conj}

\section{Acknowledgements}
V.M.~Buchstaber provided an invaluable assistance on all stages of the conducted research and preparation of the present paper. I express gratitude to the referee for careful reading of the paper and for his valuable comments and detailed suggestions which helped to give a more accurate and faithful exposition. I am grateful to A.A.~Ayzenberg and T.E.~Panov for different valuable remarks, N.Y.~Erokhovets for the communication on the combinatorics of convex polytopes and V.A.~Kirichenko for the remarks on the volume polynomial. Being the ``Young Russian Mathematics'' award winner, I would like to thank its sponsors and jury.

\end{document}